\documentclass[11pt, reqno]{amsart}
\usepackage{amssymb,amsfonts,amsxtra,amsmath,enumerate,tikz-cd}
\usepackage{dsfont,mathrsfs}
\usepackage[all]{xy}
\usepackage{lineno}
\DeclareFontFamily{OT1}{rsfs}{}
\DeclareFontShape{OT1}{rsfs}{n}{it}{<-> rsfs10}{}
\DeclareMathAlphabet{\mathscr}{OT1}{rsfs}{n}{it}

\newtheorem{theorem}{Theorem}[section]
\newtheorem{lemma}[theorem]{Lemma}
\newtheorem{corol}[theorem]{Corollary}
\newtheorem{prop}[theorem]{Proposition}

{\theoremstyle{definition} }
{\theoremstyle{remark} \newtheorem{remark}[theorem]{Remark}
\newtheorem{example}[theorem]{Example}

\newcommand{\Asr}{{\mathscr{A}}}
\newcommand{\Osr}{{\mathscr{O}}}
\newcommand{\Pbb}{{\mathbb{P}}}
\newcommand{\Abb}{{\mathbb{A}}}

\DeclareMathOperator{\td}{td}
\DeclareMathOperator{\Res}{Res}

\title[on the explicit calculation of Hirzebruch-Milnor classes]{on the explicit calculation of Hirzebruch-Milnor classes of hyperplane arrangements}
\author{Xia Liao, Youngho Yoon}

\address{
Department of Mathematical Sciences, 
Huaqiao University, 
Chenghua North Road 269, 
Quanzhou, Fujian, China
}
\email{xliao@hqu.edu.cn}

\address{
Department of Mathematical Sciences, 
Seoul National University, 
GwanAkRo 1,  Gwanak-Gu,  
Seoul 08826, Korea
}
\email{nsyyh@snu.ac.kr}

\begin{document}
\maketitle

\begin{abstract}The Hirzebruch-Milnor class is given by the difference between the homology Hirzebruch characteristic class and the virtual one. It is known that the Hirzebruch-Milnor class for a certain singular hypersurface can be calculated by using the Hodge spectrum of each stratum of singular locus. So far there is no explicit calculation of this invariant for any non-trivial examples, and we calculate this invariant by two different ways for low dimmensional hyperplane arrangements.
\end{abstract}

\section{introduction}

Let $X$ be a complex algebraic or analytic variety, and denote the singular locus of $X$ by $\Sigma$. Let $\mathbf{H}_{k}(-)$ be the Borel-Moore homology $H_{2k}^{BM}(-,\mathbb{Q})$ or the rationalised Chow group $CH_k(-)_\mathbb{Q}$, depending on the  algebraic or analytic context with which one works. Under this general framework, one can define the Milnor class, $M(X) \in \mathbf{H}_{k}(\Sigma)$ by 
\begin{equation}\label{mil}
M(X) = c^{F}(X) - c_{\textup{SM}}(X)
\end{equation}  
where $c_{\textup{SM}}(X)$ is the Chern-Schwartz-MacPherson class of $X$ and $c^{F}(X)$ is the Chern-Fulton class of $X$. When $X$ is a singular hypersurface in a $n$-dimensional complex manifold $Y$, the theories about $M(X)$ are now classical. In this case, the singular hypersurface $X$ has a virtual tangent bundle $[TX^{vir}] = [TY\vert_X] - [N_{X\slash Y}]$, and $c^F(X)$ can be written as $c(TX^{vir}) \cap [X]$. Moreover, if $X$ has only isolated singularities, then 
\begin{equation}\label{iso-mil}
M(X) = (-1)^{n-1}\sum_{x \in \Sigma} \mu_{x}[x]
\end{equation} 
where $\mu_x$ is the Milnor number at $x$. Presumably it is after this property that the class $M(X)$ is named.

Let us stick to the last case for a while and take a local equation $f$ for $X$ near an isolated singular point $p$. We know that the Milnor fibre $F_{f,p}$ at $p$ is homotopic equivalent to a bouquet of $\mu_p$ number of $(n-1)$-spheres $S^{n-1} \vee \ldots \vee S^{n-1}$, hence especially the reduced cohomologies of the Milnor fibre $\tilde{H}^i(F_{f,p},\mathbb{Q})$ all vanish except when $i=n-1$, and the Milnor number $\mu_p=\dim \tilde{H}^{n-1}(F_{f,p},\mathbb{Q})$. There is nothing more we can say about the cohomology group $H^{n-1}(F_{f,p},\mathbb{Q})$ as far as the Milnor class is the only consideration. However, according to the work of Steenbrink \cite{MR0485870}, there exists a limit mixed Hodge structure on $H^{n-1}(F_{f,p},\mathbb{Q})$, so the delicate information of the two filtrations on $H^{n-1}(F_{f,p},\mathbb{Q})$ is completely missing in $\mathcal{M}(X)$. In the more general setting where the hypersurface $X$ has non-isolated singularities, there is a similar situation where each cohomology group $H^{i}(F_{f,p},\mathbb{Q})$ is equipped with a mixed Hodge structure (\cite{MR2393625} Lemma-Definition 12.5). If one wants to find a homology class supported on $\Sigma$ which reflects certain Hodge theoretic properties of the Milnor fibres, it does not seem to us that $M(X)$ is the ideal candidate, though we do not know any reference which conspicuously addressed this issue. 

In the light of the recent emergence of the Hirzebruch class in singularity theory, the quest of finding a homology class describing Hodge theoretic properties of Milnor fibres is answered affirmatively. We now have a Hirzebruch-Milnor class at our disposal, introduced by L. Maxim, M. Saito and J. Sch\"{u}ermann in \cite{MR3053711}, for complete intersections in smooth projective varieties. Its formal definition is rather difficult, and we will list a few essential ingredients in \S\ref{general} to give the readers a morsel of its taste. However, in the simpler case that the complete intersection $X$ is acutally a hypersurface (i.e. there is a smooth ambient space $Y$ containing $X$ and $\textup{codim}_XY = 1$), the conceptual meaning of the Hirzebruch-Milnor class $M_y(X) \in \mathbf{H}_{\bullet}(\Sigma)[y]$ is attainable by considering the Milnor fibers on $\Sigma$. Indeed, we can first fix a generic hyperplane $X'$ (supposing for simplicity that $Y = \mathbb{P}^n$). Let $f$ be the defining equation of the affine hypersurface $X \setminus X' \subset \mathbb{A}^n$. it is well known that there exists a Whitney stratification $\mathcal{S} $ on $\Sigma \setminus X'$ such that different points chosen from the same stratum $S \in \mathcal{S}$ have diffeomorphic Milnor fibres, which we denote by $F_{f,S}$. For any $i \in \mathbb{Z}$, the cohomology groups of the Milnor fibres $H^i(F_{f,S}, \mathbb{Q})$ vary on $S$, forming a local system, and it canonically underlies admissible variations of the mixed Hodge structures on $H^i(F_{f,S}, \mathbb{Q})$. Roughly speaking, the Hirzebruch-Milnor class $M_y(X)$ provides information about the mixed Hodge structures on the cohomologies of these Milnor fibres and their variations within the strata. In particular, when the hypersurface has only isolated singularities, the contribution to $M_y(X)$ from a singularity $x \in \Sigma$ equals the $\chi_y$-genus of the Milnor fibre $F_x$ (\cite{MR3053711} Corollary 2), recovering \eqref{iso-mil} by setting $y=-1$. 

Though the class $M_y(X)$ has the appealing properties described above, the actual computation of the class is quite challenging. Our main motivation for this paper is to work out explicit expressions of Hirzebruch-Milnor classes in some simple but non-trivial cases. Very specifically, we have computed the Hirzebruch-Milnor classes for any reduced projective hyperplane arrangements inside $\mathbb{P}^2$ and $\mathbb{P}^3$ with two different algorithms. We would like to take the opportunity to explain the basic ideas here.

\vspace{0.5cm}
\noindent{\bf{Algorithm 1}}

The first algorithm we use is based on the following theorem.

\begin{theorem}\label{hmdiff}(\cite{MR3053711} Theorem 1)
Let $Y$ be a smooth projective variety, let $\mathcal{L}$ be a very ample line bundle on $Y$, and let $X$ be a hypersurface (not necessarily reduce) in $Y$ defined by $s\in \Gamma(Y,\mathcal{L}^{\otimes m})$ for some positive integer $m$. Then the following formula holds
\begin{equation}\label{hm1}
T^{vir}_{y*}(X) - T_{y*}(X) = (i_{\Sigma,X})_*M_y(X) \in \mathbf{H}_{\bullet}(X)[y],
\end{equation}
\end{theorem}

Here $T_{y*}(X)$ and $T^{vir}_{y*}(X)$ are the Hirzebruch class and the virtual Hirzebruch class of $X$ respectively. The map $i_{\Sigma,X}: \Sigma \to X$ is the closed embedding. (We will use freely the notation $i_{A,B}$ to indicate a closed embedding $A \to B$.) The classes $T_{y*}(X)$ and $T^{vir}_{y*}(X)$ both live in $\mathbf{H}_{\bullet}(X)[y]$, and their definitions will be briefly reviewed in \S\ref{general}. We only mention here that the Hirzebruch class transformation $T_{y*}$ was first introduced in \cite{MR2646988} as a natural transformation, unifying the Chern class transformation of Schwartz-MacPherson, the Todd class transformation of Baum-Fulton, and the $L$-class transformation of Cappell-Shaneson. With this unification point of view in mind, it is not a surprise to see that the shape of equation \eqref{hm1} resembles that of \eqref{mil}. In fact, by specialising the parameter $y$ to $-1$, the classes $T_{y*}(X)$ and $T^{vir}_{y*}(X)$ become the CSM class and the virtual Chern class thereof. However, we remind the reader that equation \eqref{hm1} is \textit{not} the definition of the Hirzebruch-Milnor class, but is rather a property of the class.

We apply theorem \ref{hmdiff} in the case that $Y = \mathbb{P}^n$ and $X = \mathscr{A}$ a reduced hyperplane arrangement of degree $m$. Instead of calculating $M_y(\mathscr{A})$, we calculate $(i_{\Sigma,\mathbb{P}^n})_*M_y(\mathscr{A})$ first. This amounts to calculating $(i_{\mathscr{A},\mathbb{P}^n})_*(T_y(\mathscr{A}))$ and $(i_{\mathscr{A},\mathbb{P}^n})_*(T_y^{vir}(\mathscr{A}))$ individually, according to theorem \ref{hmdiff}. The former follows quickly from a result about the motivic Chern class of $\mathscr{A}$ in \cite{Kdefect}, and the latter follows straight from the definition of virtual Hirzebruch classes. In the end, we obtain $(i_{\Sigma,\mathbb{P}^n})_*M_y(\mathscr{A})$, which is not quite what we aim to get in the beginning. However, there is not much lost of information, because by proposition \ref{hmlgsigma} $(i_{\Sigma,\mathbb{P}^n})_*: \mathbf{H}_k(\Sigma) \to \mathbf{H}_k(\mathbb{P}^n)$ is an isomorphism except for $k=\dim \Sigma$. 

This algorithm is applicable to any reduced projective hyperplane arrangement. In this paper we only deal with arrangements in $\mathbb{P}^2$ and $\mathbb{P}^3$ because the computation becomes significantly more difficult as $\dim \mathscr{A}$ increases.

\vspace{0.5cm}
\noindent{\bf{Algorithm 2}}

The second algorithm we use for computing $M_y(\mathscr{A})$ is based on the study of local systems of cohomologies of Milnor fibres on certain Whitney stratification of $\Sigma$. 

One dominating scene in the study of characteristic classes is the local-global phenomenon. Namely, a globally defined class can usually be written as the summation of local contributions. In the context of the Hirzebruch Milnor class, this leads one to ask whether it is possible to recover the class from the information on a Whitney stratification of $\Sigma$.

Very generally, for a hypersurface $X$, it is possible to distinguish the contributions from different strata to $M_y(X)$, at the cost of demanding additional conditions on $X$ (see condition (a) - (d) in \cite{Maxim2016}). When these extra conditions are met, there is a Whitney stratification of $\Sigma$ (instead of a Whitney stratification of $\Sigma \setminus X'$), and a decomposition formula for $M_y(X)$ in terms of the contribution from each stratum, which turns out to be determined by the Hodge spectrum of the stratum $S$ under consideration and local systems describing the variation of mixed Hodge structure on the cohomology of certain Milnor fibre on $S$. The formula was proven originally in \cite{Maxim2016}, and we recorded a slight variation of it which is adapted to hyperplane arrangements in \eqref{spec1} and \eqref{spec2} below. At first sight, that formula may look long and unintuitive, however we find that its mathematical content is surprisingly simple and elegant in contrast to its intimidating shape. By what we have remarked earlier in the introduction, the contribution from the stratum $S$ to $M_y(X)$ is a ``measure" of a local system, which underlies the variation of mixed Hodge structure on the cohomology of certain Milnor fibre. When conditions (a) - (d) loc.cit. are met, this local system breaks into the direct sum of certain rank one local systems. Each such rank one local system is indexed by a rational number $\alpha$ encoding information about both an eigenvalue of the local system monodromy and the Hodge filtration. The multiplicity of this rank one local system in the direct sum decomposition is given by $n_{f,S,\alpha}$ (in the notation of \cite{Maxim2016}), and the contribution of each copy of this rank one local system to $M_y(X)$ is given by the $\textup{td}_{(1+y)*}$ part of the formula.

It should be noted that similar results were known long time ago for Milnor classes \cite{MR1795550}. The decomposition formula for $M_y(X)$ is a Hodge-theoretic enhancement of the earlier result about the Milnor class.

Fortunately for us, the additional conditions (a) - (d) mentioned before are met for any projective hyperplanes arrangements $\mathscr{A}$ \cite{Maxim2016}. In fact, the decomposition formula \eqref{spec1} and \eqref{spec2} was applied in the case $X = \mathscr{A}$ in loc.cit., and an algorithm to compute the $\textup{td}_{(1+y)*}$ part was given. This algorithm together with an earlier result about the combinatorial nature of the Hodge spectra of hyperplane arrangements proves that $M_y(\mathscr{A})$ is combinatorial, meaning that it is completely determined by the structure of the intersection lattice of $\mathscr{A}$. The advantage of this approach is that, the contributions to $M_y(\mathscr{A})$ from different strata are distinguished. But the drawback is that, the algorithm involves a sequence of blow-ups, so it can hardly be performed in practice as soon as $\dim \mathscr{A} \geq 4$. Moreover, there is no explicit formula for the Hodge spectrum of $\mathscr{A}$ yet, when $\dim \mathscr{A} \geq 4$. 

We have taken the endeavour to use the decomposition formula to compute $M_y(\mathscr{A})$ when $\mathscr{A} \subset \mathbb{P}^2$ or $\mathscr{A} \subset \mathbb{P}^3$. The result is recorded in theorem \ref{HMP2} and theorem \ref{HMP3}. One can see that the $\mathscr{A} \subset \mathbb{P}^3$ case is already nontrivial.

\vspace{0.5cm}
To summarise, our calculation gives the following explicit expressions of the Hirzebruch-Milnor classes of reduced projective hyperplane arrangements inside $\mathbb{P}^2$ and $\mathbb{P}^3$. 
\begin{theorem}
Let $\mathscr{A}$ be a reduced projective hyperplane arrangement in $\mathbb{P}^2$ or $\mathbb{P}^3$.
\begin{enumerate}[(i)]
\item For those $\mathscr{A} \subset \mathbb{P}^2$, we have
\begin{equation*}
M_y(\Asr)=\sum_{P\in \Sigma} \left( \frac{m_P(m_P-1)}{2}y-\frac{(m_P-1)(m_P-2)}{2} \right) [P]
\end{equation*}
where $m_P$ is the number of lines passing through the singular point $P$. This is theorem \ref{HMP2}. See also the discussion in \S \ref{Kline}.
\item For those $\mathscr{A} \subset \mathbb{P}^3$, the result derived from algorithm 1 is recorded in theorem \ref{HMP3'}, and the result derived from algorithm 2 is recorded in theorem \ref{HMP3}.
\end{enumerate}
\end{theorem}

Finally we wish to compare the $M_y(\mathscr{A})$ obtained from these two algorithms. We can see that the computations give the matching results when $\mathscr{A} \subset \mathbb{P}^2$. This ought also to happen for those $\mathscr{A} \subset \mathbb{P}^3$. By comparing the degrees of $M_y(\mathscr{A})$ in theorem \ref{HMP3} and \ref{HMP3'}, one can obtain a combinatorial formula. It seems that one may prove this combinatorial formula directly by some elementary manipulation of power series coefficients, but by far we are not successful in doing this. We have experimented these formulas in example \ref{exP3} and \ref{exP3'}, and indeed we get compatible results.

This paper arises from our attempt to understand the references \cite{MR3053711}, \cite{Maxim2016}. In these references, only the general theory of Hirzebruch-Milnor class were established, but no explicit calculations were given. We hope that our modest results can help to bridge the gap between theoretical interests and practical applications. 

\noindent
{\bf Acknowledgements}: The main part of the paper was written when both authors were postdoc research fellows at KIAS, another part of the paper was written when the first named author was visiting the second named author at Seoul National University. The first named author is supported by Huaqiao University research fund 600005-Z18Y0025, and he wants to thank SNU for providing a very hospitable research environment. The second named  author is supported by BK21 PLUS SNU Mathematical Sciences Division and the National Research Foundation of Korea(NRF) grant funded by the Korea government(MSIP) (No. NRF-2017R1C1B1005166).

\section{general background on Hirzebruch-Milnor classes}\label{general}

We summarise the definitions and the main conclusions we need about Hirzebruch classes and Hirzebruch-Milnor classes.  Our main reference for the Hirzebruch class is \cite{MR2646988}, and our main references for the Hirzebruch-Milnor class are \cite{MR3053711} and \cite{Maxim2016}. In a few occasions the notations in these references are chosen differently. When this happens, we follow the choice of notations in \cite{Maxim2016}. In this paper, all varieties are defined over $\mathbb{C}$. 

\subsection{Hirzebruch classes of singular spaces}

Let us fix some notations first. Given an algebraic variety $X$, the relative Grothendieck group of varieties over $X$ is denoted by $K_0(var/X)$, and the $K$-group of coherent sheaves on $X$ is denoted by $K_0(X)$. 
 
The motivic Hirzebruch class transformation $T_{y*}$ can be defined as the composition of natural transformations
\begin{equation*}
mC_{y} : K_0(var/-) \to K_0(-)\otimes \mathbb{Z}[y]
\end{equation*}
and 
\begin{equation*}
\td_{(1+y)*} : K_0(-)\otimes \mathbb{Z}[y] \to \mathbf{H}_{\bullet}(-)[(1+y)^{-1},y].
\end{equation*}
Let us explain the two natural transformations one by one. The natural transformation $mC_y$ is called the motivic Chern class transformation. We have the following facts concerning $mC_y$.
\begin{enumerate}[(i)]
\item For a smooth and purely dimensional variety $X$, $mC_y([X \xrightarrow{id} X]) =\sum_p [\Omega_X^p]y^p$. 
\item For a singular variety $X$, there is an expression of $mC_y([X \xrightarrow{id} X])$ similar to that of case (i), with the de Rham complex $\Omega^{\bullet}_X$ for a smooth $X$ replaced by the Du Bois complex of a singular $X$. See \cite{MR3417881} \S 2.1.2. 
\item Given a morphism $f: Y \to X$ of algebraic varieties, to calculate $mC_y([Y \to X])$, we may assume $f$ is proper, because $K_0(var/-)$ is generated by proper morphisms. Then $mC_y([Y \to X]) = f_*(mC_y([Y \to Y])$. Here $f_*: K_0(Y) \to K_0(X)$ extends by linearity to a map $K_0(Y)[y] \to K_0(X)[y]$, which we still denote by $f_*$.
\end{enumerate}

Given a variety $X$, the Todd class transformation $\td_*: K_0(X) \to \mathbf{H}_{\bullet}(X)$ is the key ingredient in the classical theory of singular Riemann-Roch theorem (denoted by $\tau$ in \cite{MR1644323} chapter 18). It extends by linearity to $K_0(X)[y] \to \mathbf{H}_{\bullet}(X)[y]$, for which we still use the notation $\td_*$. The modified Todd class transformation $\td_{(1+y)*}$ is the composition of $\td_*$ with a normalisation process in $\mathbf{H}_{\bullet}(X)[y]$. For a class $\alpha \in \mathbf{H}_{\bullet}(X)[y]$, we first write 
\begin{equation*}
\alpha = \sum_i \alpha_i
\end{equation*}
where $\alpha_i \in \mathbf{H}_{i}(X)[y]$. The normalisation process then takes $\alpha$ to 
\begin{equation}\label{normalisation}
\sum_i (1+y)^{-i}\alpha_i \in \mathbf{H}_{i}(X)[(1+y)^{-1},y]
\end{equation}

As stated earlier, we define the Hirzebruch class transformation $T_{y*}$ by 
\begin{equation*}
T_{y*} = \td_{(1+y)*} \circ mC_y.
\end{equation*}
We also define the unnormalised Hirzebruch class transformation $\tilde{T}_{y*}$ by $\tilde{T}_{y*} = \td_{*} \circ mC_y$. We will use $T_{y*}(X)$ as the abbreviated notation for $T_{y*}([X \to X])$. Our description of $T_{y*}$ implies that it is a natural transformation from $K_0(var/-)$ to $\mathbf{H}_{\bullet}(-)[(1+y)^{-1},y]$, but one can show that it actually lands in the smaller group $\mathbf{H}_{\bullet}(-)[y] \subset \mathbf{H}_{\bullet}(-)[(1+y)^{-1},y]$.

The transformation $mC_y$ factorises through the functor of Grothendieck group of mixed Hodge module $K_0(MHM(-))$. Indeed, we have two natural transformations $mH: K_0(var/-) \to K_0(MHM(-))$ and $\textup{DR}_y: K_0(MHM(-)) \to K_0(-)[y,y^{-1}]$, and a commutative diagram

\begin{equation*}
\begin{tikzcd}[column sep=small]
K_0(var/-) \arrow[rd,"mH"] \arrow[rr,"mC_y"] & &  K_0(-)[y,y^{-1}]  \\
                                            & K_0(MHM(-)) \arrow[ru,"\textup{DR}_y"] &
\end{tikzcd}.
\end{equation*}

The following facts and remarks might be useful for the readers.
\begin{enumerate}[(i)]
\item The notation $\textup{DR}_y$ is used in \cite{Maxim2016} while the same transformation is denoted by $gr^{F}_{-*}DR$ in \cite{MR2646988}. The Hodge filtration $F^{*}$ is an increasing filtration in \cite{MR2646988}. The Hodge filtration $F_{*}$ on the de Rham complex in \cite{Maxim2016} is a decreasing filtration. They are related by $F^{-*} = F_{*}$.
\item One can show that the composition $\textup{DR}_y \circ mH : K_0(var/-) \to K_0(-)[y,y^{-1}]$ lands in the smaller group $K_0(-)[y] \subset K_0(-)[y,y^{-1}]$. (\cite{MR2646988} Corollary 5.1) 
\item For any variety $X$, there is a canonical element $\mathbb{Q}_{h,X} \in MHM(X)$ whose underlying $\mathbb{Q}$-complex is the constant sheaf $\mathbb{Q}_X$, serving as the ``constant Hodge function". For a morphism $f: Y \to X$ we have $mH([Y\to X]) = [f_!\mathbb{Q}_{h,Y}]$. In particular, the commutativity of the diagram above implies that $mC_Y([X \to X]) = \textup{DR}_y([\mathbb{Q}_{h,X}])$ and $T_{y*}(X) = \td_{(1+y)*}\circ \textup{DR}_y([\mathbb{Q}_{h,X}])$. The notation $\textup{DR}_y[X] := \textup{DR}_y([\mathbb{Q}_{h,X}])$ is also used in \cite{Maxim2016}.
\item The commutative diagram above implies that, apart from the motivic Hirzebruch class transformation we can also define a Hodge theoretic Hirzebruch class transformation by $\td_{(1+y)*}\circ \textup{DR}_y : K_0(MHM(-)) \to \mathbf{H}_{\bullet}(-)[y]$. The Hodge theoretic Hirzebruch class transformation is also denoted by $T_{y*}$. We hope that the explanations given here will clarify all confusions regarding the abuse of notation. More generally, given $\mathcal{M}^{\bullet} \in D^bMHM(X)$, there is an associated class $[\mathcal{M}^{\bullet}] \in K_0(MHM(X))$. So one can define $T_{y*}(\mathcal{M}^{\bullet}) = \td_{(1+y)*}\circ \textup{DR}_y([\mathcal{M}^{\bullet}])$.
\end{enumerate}

\subsection{Virtual Hirzebruch class}

Let $X$ be a complete intersection in a smooth ambient space $Y$. We can define the virtual Hirzebruch class of $X$ by mimicking the definition of the Hirzebruch class of a smooth space.

In fact, if $X$ is smooth, by what we have discussed earlier its Hirzebruch class is given by
\begin{equation*}
T_{y*}(X) = \td_{(1+y)*}(\sum_p [\Omega^p_X]y^p).
\end{equation*}
Note that $\sum_p [\Omega^p_X]y^p$ is also the total Lambda class $\Lambda_y[\Omega^1_X]$ of $\Omega^1_X$. 

If $X$ is a complete intersection in $Y$, we define 
\begin{equation*}
\begin{split}
&\textup{DR}^{vir}_y[X] = \Lambda_y [TX^*_{vir}] \\
&T^{vir}_{y*}(X) = \td_{(1+y)*}(\textup{DR}^{vir}_y[X]) 
\end{split}
\end{equation*}
where $[TX_{vir}] = [TY\vert_{X}] - [N_{X \slash Y}]$ is the virtual tangent bundle and $[TX^*_{vir}] = [\Omega^1_Y\vert_X] - [N^*_{X \slash Y}]$ is its virtual dual. The $\lambda$-ring structure of $K_0(X)$ implies that 
\begin{equation*}
\Lambda_y [TX^*_{vir}] = \frac{\Lambda_y  [\Omega^1_Y\vert_{X}]}{\Lambda_y [N^*_{X \slash Y}]} \in K^0(X)[[y]],
\end{equation*}
but it can be shown that $\Lambda_y [TX^*_{vir}] \in K_0(X)[y]$ (\cite{Maxim2016} proposition 3.4).

For our application in \S \ref{K-theoretic-hm}, we show how to compute $(i_{X,Y})_*T^{vir}_{y*}(X)$ when $X$ is a hypersurface in $Y$. This equals 
\begin{equation*}
\td_{(1+y)*}(i_{X,Y})_*\textup{DR}^{vir}_y[X]
\end{equation*}
by the naturality of $\td_{(1+y)*}$. For a hypersurface $X$, we have $N_{X \slash Y} \cong \mathscr{O}_Y(X)\vert_X$ and therefore
\begin{equation}\label{virhyper}
\begin{split}
(i_{X,Y})_*\textup{DR}^{vir}_y[X]  &=(i_{X,Y})_* (\Lambda_y [TX^*_{vir}]) \\
                        &= \frac{\Lambda_y  [\Omega^1_Y]}{\Lambda_y [\mathscr{O}_Y(-X)]} \otimes [\mathscr{O}_X] \\
                        &= \frac{\Lambda_y  [[\Omega^1_Y]}{\Lambda_y [\mathscr{O}_Y(-X)]} \otimes \Big([\mathscr{O}_Y] - [\mathscr{O}_Y(-X)]\Big).
\end{split}
\end{equation}
This computation also shows that if $\tilde{X}$ is a smooth hypersurface in $Y$ linearly equivalent to $X$, then
\begin{equation*}
\begin{split}
&(i_{X,Y})_*\textup{DR}^{vir}_y[X]  = (i_{\tilde{X},Y})_*\textup{DR}^{vir}_y[\tilde{X}] = (i_{\tilde{X},Y})_*\textup{DR}_y[\tilde{X}] \\
& (i_{X,Y})_*T^{vir}_{y*}(X) = (i_{\tilde{X},Y})_*T_{y*}(\tilde{X}).
\end{split}
\end{equation*}

\subsection{Hirzebruch-Milnor classes}

The setup for the definition of Hirzebruch-Milnor class is the following. Let $Y^{(0)}$ be a smooth projective variety and let $\mathcal{L}$ be a very ample line bundle on $Y^{(0)}$. Let $s_i \in \Gamma(Y,\mathcal{L}^{\otimes a_i})$ where $\{a_i\} \ (i = 1 ,\dots ,k)$ be a decreasing sequence of positive integers. Finally let $X = \cap s^{-1}_i(0)$ be a complete intersection of codimension $k$, and let $\Sigma$ be the singular locus of $X$. Then one can define the Hirzebruch-Milnor class of $X$, denoted by $M_y(X)$, along the following line:
\begin{enumerate}[(i)]

\item The sections $s_{i} \in \Gamma(Y^{(0)}, \mathcal{L}^{a_i})$ which define $X$ can be chosen inductively, so that they satisfy the conclusion of proposition 3.2 of \cite{MR3053711}. Define 
\[
Y^{(i)} = \bigcap_{1\leq t \leq i} s_t^{-1}(0) 
\]
and let $\Sigma_{Y^{(i)}}$ be the singular locus of $Y^{(i)}$. Note that $X = Y^{(k)}$ in this notation.

\item Choose generic sections $s'_i \in \Gamma(Y^{(0)}, \mathcal{L}^{a_i})$ inductively, such that certain transversality conditions among the $s'$'s, the $Y^{(i)}$'s, and algebraic Whitney stratifications of $\Sigma_{Y^{(i)}}$'s are satisfied. 

\item One can define a mixed Hodge module $\mathcal{M}(s'_1,\ldots,s'_k)$ supported on $\Sigma = \Sigma_{Y^{(k)}}$. This step involves first defining a mixed Hodge module $\mathcal{M}'(s'_1,\ldots,s'_k)$ by iterated applications of certain vanishing cycle functors to $\mathbb{Q}_{h,\mathcal{Z}_T}$, where space $\mathcal{Z}_T$ is constructed by using $s$'s and $s'$'s, and showing that there is a canonical injection $\mathbb{Q}_{h,X}[\dim X] \to \mathcal{M}'(s'_1,\ldots,s'_k)$, and finally letting $\mathcal{M}(s'_1,\ldots,s'_k)$ be the cokernel of this injection.

\item The Hirzebruch-Milnor class $M_y(X)$ is defined to be the image of $\mathcal{M}(s'_1,\ldots,s'_k)$ under the Hirzebruch class transformation $T_{y*}$. One can show that $M_y(X)$ is independent on the choices of $s'$'s. It may still depend on the choices of $s$'s, but for a sufficiently general choice of the $s$'s, the class $M_y(X)$ is the same. Hence it is well defined. 

\end{enumerate}

One can see that the formal definition of $M_y(X)$ is extremely complicated. It is hopeless to compute it by appealing to the definition. In this paper, we use two different methods to compute $M_y(\mathscr{A})$ for a reduced hyperplane arrangement $\mathscr{A}$ in $\mathbb{P}^2$ or $\mathbb{P}^3$. One method is based on the study of Hodge spectra of $\mathscr{A}$ and certain local systems on a Whitney stratification of the singular locus of $\mathscr{A}$. This method was explained in \cite{Maxim2016}, but to our knowledge no one has ever tested it on any concrete example. In \S \ref{spec}, the theoretical basis of this method is reviewed and some genuine computations based on this method will be given. Another method for computing $M_y(\mathscr{A})$ is based on directly applying theorem \ref{hmdiff}. Although this method only allows us to obtain the image of $M_y(\mathscr{A})$ in the homology of the ambient projective space, not much information is lost in the image. Detailed explanations and relevant computations are given in \S \ref{K-theoretic-hm}.

\section{Computations of Hirzebruch-Milnor classes for reduced hyperplane arrangements}\label{spec}

\subsection{Algorithm for any reduced projective hyperplane arrangements}
For any projective hypersurface $X$ satisfying conditions (a) - (d) in \cite{Maxim2016}, its Hirzebruch-Milnor class $M_y(X)$ can be assembled from local contributions. Here ``local" means that there exists a Whitney stratification of $\Sigma$, and on each stratum a homology class can be specified as the contribution from this stratum to the whole. The decomposition formula is given in \cite{Maxim2016} theorem 1.1. A very important case in which conditions (a) - (d) hold is \textit{any} (not necessarily reduced!) hyperplane arrangements $\mathscr{A} \subset \mathbb{P}^n$. An algorithm was given in loc.cit. \S 4 concerning the computation of $M_y(\mathscr{A})$. In this section, we summarise the main steps of this algorithm for \textit{reduced} hyperplane arrangements. While keeping the notations as much the same as possible in loc.cit., we commit ourselves to present the algorithm in a more user-friendly way.

Let $\mathscr{A}=\bigcup\limits_{i=1}^{m}H_i$ be a reduced hyperplane arrangement in $\mathbb{P}^n$, with $H_i$ $(i=1,\cdots,m)$ the distinct hyperplanes. Let $L(\mathscr{A})$ be the intersection lattice of $\mathscr{A}$ and let $\Sigma$ be the singular locus of $\mathscr{A}$.  To decompose $M_y(\mathscr{A})$, we describe a disjoint union decomposition of $\Sigma$ and certain Milnor fibers on each union piece in the following steps. 

\begin{enumerate}

\item We choose a generic hyperplane $H'_j$ for each $j \in [1,n+1]$. We may choose the coordinate system $(x_1 : \ldots : x_{n+1})$ on $\mathbb{P}^n$ so that $H'_j = \{ x_j = 0\}$. The generic condition here guarantees that each $H'_j$ intersects elements in $L(\mathscr{A})$ as transversal as possible.

\item Set $\Sigma_1=\Sigma$ and $\Sigma_j =  (\bigcap\limits_{k<j}H'_k) \cap \Sigma$ for $j \in [2,n+1]$. We can write $\Sigma$ as a disjoint union
\begin{equation*}
\Sigma = \coprod_{j=1}^{n+1} (\Sigma_j \setminus H'_j).
\end{equation*}
We have $\Sigma_j \setminus H'_j \subset (\mathbb{P}^n \setminus H'_j)$ and $\mathbb{P}^n \setminus H'_j = \mathbb{C}^n$ with coordinates $\frac{x_1}{x_j} \ldots \hat{\frac{x_j}{x_j}} \ldots \frac{x_{n+1}}{x_j}$. Note that $\Sigma_j \setminus H'_j$ is a union of equidimensional affine linear subspaces of $\mathbb{C}^n$ and $\dim (\Sigma_j \setminus H'_j) = (n-2)-(j-1) = n-1-j$. Therefore the disjoint union above can be taken over the index set $\{1,2,\ldots,r\}$ where we let $r= n-1$.

\item Set $f_j=\left( f(0,\cdots,0, x_j,\cdots,x_{n+1})+x_1^m+\cdots+x_{j-1}^m\right)/x_j^m$ in $\mathbb{P}^n\setminus H'_j = \mathbb{C}^n$. One can show that the singular locus of the affine hypersurface defined by $f_j = 0$ (no longer a hyperplane arrangement!) is $\Sigma_j \setminus H'_j$. 

\end{enumerate}

Let $\phi_{f_j}$ be the vanishing cycle functor associated to the polynomial $f_j$ in $\mathbb{P}^n\setminus H'_j$. We can now state the first step in the decomposition of $M_y(\mathscr{A})$.

\begin{theorem}(\cite{Maxim2016} equation (2), or \cite{2016arXiv160602218M} theorem 2)
\begin{equation}\label{spec1}
M_y(\Asr) = \sum_{j=1}^{r} T_{y*}\left( \left( i_{\Sigma_j\setminus H'_j,\Sigma} \right)_! \phi_{f_j} \mathbb{Q}_{h,\mathbb{P}^n\setminus H'_j}\right)\in \mathbf{H}_{\bullet}(\Sigma)[y].
\end{equation}
In this expression, $ \phi_{f_j} \mathbb{Q}_{h,\mathbb{P}^n\setminus H'_j}$ is a mixed Hodge module up to a shift of complex on $\Sigma_j \setminus H'_j$ such that its underlying $\mathbb{Q}$-complex is the vanishing cycle $\phi_{f_j} \mathbb{Q}_{\mathbb{P}^n\setminus H'_j}$. 
\end{theorem}

To understand $M_y(\mathscr{A})$, we still need to know the structure of $\mathbf{H}_{\bullet}(\Sigma)$ and how to compute $T_{y*}\left( \left( i_{\Sigma_j\setminus H'_j,\Sigma} \right)_! \phi_{f_j} \mathbb{Q}_{h,\mathbb{P}^n\setminus H'_j}\right)$ for a fixed $j$. 

To answer these questions, We first describe a Whitney stratification of $\Sigma_j \setminus H'_j$ and introduce some notations.

\begin{enumerate}[(i)]

\item The equation $f(0,\cdots,0, x_j,\cdots,x_{n+1}) = 0$ defines a hyperplane arrangement $\mathscr{A}_j$ in $Y_j:=\bigcap\limits_{k<j} H'_k \cong \mathbb{P}^{n+1-j}$. One can see that $\Sigma_j$ is the singular locus of $\mathscr{A}_j$. The intersection lattice $L(\mathscr{A}_j)$ induces canonical stratifications on both $\Sigma_j$ and $\Sigma_j \setminus H'_j$. The latter stratification is denoted by $\mathcal{S}_j$.

\item For each $S\in \mathcal{S}_j$ we introduce the index set $I(S)=\{i|S\subset H_i\}$. We have $\bar{S}=\left(\bigcap\limits_{k<j} H'_k \right) \cap \left( \bigcap\limits_{i\in I(S)}H_i \right)$ and $S=\bar{S}\setminus\left(\bigcup\limits_{i\not\in I(S)} H_i\cup H'_j\right)$

\item The set $\{ \bar{S} \ | \ S\in \mathcal{S}_j \}$ is the set of elements in $L(\mathscr{A}_j)$ which has codimension in $\mathbb{P}^{n+1-j}$ bigger than $1$. (The codimension 1 elements in $L(\mathscr{A}_j)$ are exactly the hyperplanes in $L(\mathscr{A}_j)$.) Let $\mathcal{S}_j^{(c)} = \{ S\in \mathcal{S}_j \ | \ \textup{codim}_{Y_j} S =c \}\subset \mathcal{S}_j$ be the set of codimension $c$ elements. By default $c \geq 2$. We have $\dim S = n+1-j-c$ for any $S \in \mathcal{S}_j^{(c)}$.

\end{enumerate}

We can now describe the homology group $\mathbf{H}_{\bullet}(\Sigma)$.
\begin{prop}(\cite{Maxim2016} equation (48))\label{hmlgsigma} 
For a reduced projective hyperplane arrangement $\Asr \subset \mathbb{P}^n$, we have
\begin{equation*} 
\mathbf{H}_k(\Sigma)=\begin{cases}\bigoplus_{S\in \mathcal{S}_1^{(2)}}\mathbb{Q}[\bar{S}]&\text{if $k=n-2$},\\
\mathbb{Q}&\text{if $0\leq k\leq n-3$.}
\end{cases}
\end{equation*}
\end{prop}

Next, choose an arbitrary point $x\in S \in \mathcal{S}_j$. The Hodge spectrum (or Steenbrink spectrum) of $f_j$ at $x$ is the fractional Laurent polynomial 
\begin{equation*}
Sp_{f_j, x}(t)=\sum_{\alpha \in \mathbb{Q}}n_{f_j,x,\alpha}t^{\alpha},
\end{equation*}
where $n_{f_j,x,\alpha}$ is called the spectral multiplicity of the spectral number $\alpha$. The Hodge spectrum can be defined for any holomorphic function $f$ at any $x \in f^{-1}(0)$. See for example \cite{Maxim2016} \S 2.5 for the definition under this general setting. The Hodge spectrum $Sp_{f_j, x}(t)$ is \textit{independent} of the choice of $x \in S$ for a fixed $S$, but it may depend on the choice of $S \in \mathcal{S}_j$. Thus, it is legitimate to use the notations $Sp_{S}(t):=Sp_{f_j, x}(t)$ and $n_{S,\alpha} := n_{f_j,x,\alpha}$.

To calculate $n_{S,\alpha}$, we will frequently use the following Thom-Sebastiani type formula, which was proven by M. Saito. This formula holds generally for a joint of singularities (see \cite{MR1621831} (8.7.1) for the definition of the joint of singularities).

\begin{theorem}\label{thom}(\cite{MR1621831}-II (8.10.6))
Assume that an analytic function germ $g:(\mathbb{C}^n,\ 0)\rightarrow (\mathbb{C},\ 0)$ can be written as $g(w_1,\cdots,w_n)=g_1(w_1,\cdots,w_m)+g_2(w_{m+1},\cdots,w_n)$ for $g_1:(\mathbb{C}^m,\ 0)\rightarrow (\mathbb{C},\ 0)$ and $g_2:(\mathbb{C}^{n-m},\ 0)\rightarrow (\mathbb{C},\ 0)$. Then $Sp_g(t)=Sp_{g_1}(t)Sp_{g_2}(t)$. In particular $Sp_g(t)=(-t)^{m}Sp_{g_2}(t)$ if $g_1=0$.
\end{theorem}

\begin{remark}
The reader should not confuse our notation $n_{S,\alpha} = n_{f_j,x,\alpha}$ with the notation $n_{f,S,\alpha}$ in \cite{Maxim2016} equation (41). Recall that $f_j=\left( f(0,\cdots,0, x_j,\cdots,x_{n+1})+x_1^m+\cdots+x_{j-1}^m\right)/x_j^m$ in our notation. The $f$ in the $n_{f,S,\alpha}$ notation of \cite{Maxim2016} equation (41) is our $f(0,\cdots,0, x_j,\cdots,x_{n+1})/x_j^m$ viewed as a function in $\mathbb{C}^{n+1-j}$. Therefore by theorem \ref{thom}, knowing the Hodge spectrum of the function $f(0,\cdots,0, x_j,\cdots,x_{n+1})/x_j^m$ will suffice for us to compute the Hodge spectrum of $f_j$, since $Sp_{x^m,0}(t) = \frac{t^{1/m}-t}{1-t^{1/m}}$. Moreover, the formula $n_{f,S,\alpha} =(-1)^{\dim S}n_{f^S,0,\beta}$ given by equation (41) \cite{Maxim2016} is the special case $g_1=0$ and $g_2=f^S$ in theorem \ref{thom}, since $f(0,\cdots,0, x_j,\cdots,x_{n+1})/x_j^m$ defines a cone over $S$ in a neighbourhood of $S$. 
\end{remark}

\begin{remark}
Combinatorial formulas for the Hodge spectrum of reduced hyperplane arrangements in $\mathbb{P}^n$ when $n\leq 4$ are given in \cite{MR3344208}.
\end{remark}

With the help of the spectral multiplicities, we can compute the term $T_{y*}\left( \left( i_{\Sigma,\Asr} \right)_! \phi_{f_j} \mathbb{Q}_{h,\mathbb{P}^n\setminus H'_j}\right)$ in equation \eqref{spec1}.
\begin{theorem}\label{decom}(\cite{Maxim2016} theorem 1.1) 
\begin{equation}\label{spec2}
\begin{split}
T_{y*}\left( \left( i_{\Sigma,\Asr} \right)_! \phi_{f_j} \mathbb{Q}_{h,\mathbb{P}^n\setminus H'_j}\right)=&\sum_{S\in \mathcal{S}_j} \sum_{\alpha,q}(-1)^{q+n-1}n_{S,\alpha}(\pi_{\tilde{S}})_* \td_{(1+y)*}\left[ \mathcal{E}_{S,\mathbf{e}(-\alpha), q}\right](-y)^{\lfloor n-\alpha \rfloor +q },\\
\end{split}
\end{equation}
with
\begin{equation}\label{local}
\begin{split}
&\mathcal{E}_{S,\mathbf{e}(-\alpha), q} :=\mathcal{L}_{\tilde{S},\mathbf{e}(-\alpha)}\otimes_{\Osr_{\tilde{S}}}\Omega^q_{\tilde{S}}\left( \log D_{\tilde{S}} \right).
\end{split}
\end{equation}
\end{theorem}
The notations in this theorem require some explanations. Here the divisor $D_{\tilde{S}}$ is the complement $\tilde{S}\setminus S$ of a good compactification $\tilde{S}$ of $S$, and $\mathcal{L}_{\tilde{S},\mathbf{e}(-\alpha)}$ is the canonical Deligne extension (with logarithmic residues chosen in $(0,1]$) of the rank 1 local system described in condition (b) of \cite{Maxim2016}. More precisely and in the notation of loc.cit, this rank 1 local system is an $\mathbf{e}(-\alpha)$-eigensheaf of the action of the local system monodromy on $\mathcal{H}^j_{S}$.

We introduce the notation
\begin{equation}\label{TS}
\begin{split}
T(S):=\sum_{\alpha,q}(-1)^{q+n-1}n_{S,\alpha}(\pi_{\tilde{S}})_* \td_{(1+y)*}\left[ \mathcal{E}_{S,\mathbf{e}(-\alpha), q}\right](-y)^{\lfloor n-\alpha \rfloor +q },
\end{split}
\end{equation}
so that for reduced hyperplane arrangements formula (\ref{spec1}) can be written as 
\begin{equation}\label{Tsum}
M_y(\Asr) = \sum_{j=1}^{n-1}\sum_{c\geq2}^{n+1-j}\sum_{S\in \mathcal{S}_j^{(c)}} T(S).
\end{equation}
Be aware that the function $f_j$ used for computing $n_{S,\alpha}$ for those $S\in \mathcal{S}_j^{(c)}$ changes as $j$ changes.

\begin{remark}
When $S\in \mathcal{S}_j^{(n+1-j)}$, $S$ is a point and we have $\pi_{\tilde{S}}=id_S$ and $(\pi_{\tilde{S}})_* \td_{(1+y)*}\left[ \mathcal{E}_{S,\alpha, q}\right]=[S]$. In this case $T(S)$ is completely determined by the Hodge spectrum $Sp_{S}(t)$.
\end{remark}

To give a combinatorial description of $\mathcal{L}_{\tilde{S},\mathbf{e}(-\alpha)}$ in $T(S)$, certain multiplicities are introduced in \cite{Maxim2016}. In the context of reduced hyperplane arrangements, these multiplicities are listed as follows.

\begin{enumerate}[(i)]

\item There is a total multiplicity $m$, which is the number of irreducible components of $\mathscr{A}$.

\item For each $S \in \mathcal{S}_j$ we have a multiplicity $m_S : = \# I(S)$, which is the number of hyperplanes in $\mathscr{A}$ containing the stratum $S$

\item For a pair of strata $S',S \in \mathcal{S}_j$ such that $S'\subsetneq \bar{S}$, we have a relative multiplicity $m_{S',S}:= m_{S'}-m_S$.

\item Additionally we need one more multiplicity $m_{\infty,S}:=-(m-m_S)$ for $S_{\infty}:=\bar{S}\cap H_j$.

\item To simplify our notations in the next proposition, we also define $m'_{S',S}=m_{S',S}-m_S\lfloor m_{S',S}/m_S \rfloor$ and $m'_{\infty,S}=m_{\infty,S}-m_S\lfloor m_{\infty,S}/m_S \rfloor$. Note that $m'_{S',S}/m_S =  \{m_{S',S}/m_S \}$, the decimal part of $m_{S',S}/m_S$. Similary, $m'_{S_{\infty},S}/m_S=\{m_{\infty,S}/m_S\}$. 

\end{enumerate}

A good compactification $\tilde{S}$ of $S\in \mathcal{S}_j$ is involved in the expression \eqref{spec2} and \eqref{local}. The natural compactification $\bar{S}$ is $\mathbb{P}^{k}$ where $k=\dim S$, and $S$ is the complement of a hyperplane arrangement in $\mathbb{P}^k$. Note that this hyperplane arrangement has a distinguished hyperplane at infinity $S_{\infty}=\bar{S}\cap H'_{j}$. A good compactification $\tilde{S}$ can be obtained by a sequence of blowing-ups starting from $\bar{S}$, whose centers are chosen one by one from (the strict transformation of) the elements in the set $\mathcal{S}_{j,S}:=\{\bar{S'}\ | \ S' \in \mathcal{S}_j,\  S' \subsetneq \bar{S} \text{ and } \textup{codim}_{\bar{S}} S' \geq 2 \}$ and are ordered by the increase of the dimension. The set $\mathcal{S}_{j,S}$ is a subset of the set of edges of the hyperplane arrangement $\bar{S}\setminus S$, and it is closed under intersection.  We denote the composition of blowing-ups by $\pi_{\tilde{S}}: \tilde{S}\to \bar{S}$, and denote $\pi^{-1}(\bar{S}\setminus S)$ by $D_{\tilde{S}}$. See \cite{MR2640043} \S 2 for a detailed account of this construction. This construction is different from the construction in \cite{Maxim2016}. However, it gives a good compactification. Under this construction, the irreducible component of $D_{\tilde{S}}$ corresponds to either $S_{\infty}$ or $S'\in \mathcal{S}_j$ such that $\bar{S'}\subsetneq \bar{S}$. The former is denoted by $E_{\infty,\bar{S}}$, and the latter is denoted by $E_{\bar{S'},\bar{S}}$.


With the above notations, we can restate the combinatorial formulas in \cite{Maxim2016}.
\begin{prop}(\cite{Maxim2016} equation (8) and equation (43))
 
\begin{equation}\label{local1}
\begin{split}
&\mathcal{L}_{\tilde{S},\mathbf{e}(\frac{k}{m_S})}=\mathcal{L}_{\tilde{S}}^{\otimes k}\otimes_{\Osr_\mathcal{\tilde{S}}}\Osr_\mathcal{\tilde{S}}\left(\sum_{S'\subsetneq \bar{S}}\left( \lceil k m'_{S',S}/m_S \rceil -1\right) E_{\bar{S}',\tilde{S}}+\left( \lceil k m'_{\infty,S}/m_S \rceil -1\right) E_{\infty,\tilde{S}} \right)
\end{split}
\end{equation}
and 
\begin{equation}\label{local2}
\mathcal{L}_{\tilde{S}}=\pi^{*}_{\tilde{S}}\Osr_{\bar{S}} \left(\lfloor m_{\infty,S}/m_S \rfloor \right)\otimes_{\Osr_{\tilde{S}}}\Osr_{\tilde{S}}\left( \sum_{S'\subsetneq \bar{S}}\lfloor m_{S',S}/m_S \rfloor E_{\bar{S}',\tilde{S}} \right)\\
\end{equation}
where $\Osr_{\bar{S}}\left( k \right)$ denotes the pull-back of $\Osr_{\Pbb^n } \left( k \right)$ by $\bar{S} \hookrightarrow \Pbb^n$.
\end{prop}
Here we use the identity $ \lceil \beta \rceil =-\lfloor - \beta \rfloor$ for any $\beta\in \mathbb{R}$. Using $m'_{S',S}=m_{S',S}-m_S\lfloor m_{S',S}/m_S \rfloor$, $m'_{\infty,S}=m_{\infty,S}-m_S\lfloor m_{\infty,S}/m_S \rfloor$, and $\pi^{*}_{\tilde{S}}\Osr_{\bar{S}} \left(\lfloor m_{\infty,S}/m_S \rfloor \right)=\Osr_{\tilde{S}}\left(\lfloor m_{\infty,S}/m_S \rfloor E_{\infty,\tilde{S}}\right)$, we can combine (\ref{local1}) and (\ref{local2}). This leads to the following intermediary result.

\begin{prop} 
For any $S \in \mathcal{S}_j$, we have
\begin{equation}\label{T1}
\mathcal{L}_{\tilde{S},\mathbf{e}(\frac{k}{m_S})}=\Osr_\mathcal{\tilde{S}}\left(\sum_{S'\subsetneq \bar{S}}\left( \lceil k m_{S',S}/m_S \rceil -1\right) E_{\bar{S}',\tilde{S}}+\left( \lceil k m_{\infty,S}/m_S \rceil -1\right) E_{\infty,\tilde{S}} \right)
\end{equation}
\end{prop}

To calculate $ \td_{(1+y)*}\left[ \mathcal{E}_{S,\mathbf{e}(-\alpha), q}\right]$ in $T(S)$ we use the exact sequence, 
\begin{equation}\label{T2}
\begin{split}
0 \to \Omega^1_{\tilde{S}} \to \Omega^1_{\tilde{S}}\left( \log D_{\tilde{S}} \right) \to \bigoplus_{E\subset D_{\tilde{S}}}\Osr_{\tilde{S}}(E)\to 0,
\end{split}
\end{equation}
where $E$ is an irreducible component of $D_{\tilde{S}}$. The last part can be written equivalently as 
 \begin{equation}\label{T3}
  \bigoplus_{E\subset D_{\tilde{S}}}\Osr_{\tilde{S}}(E)=\bigoplus_{S'\subsetneq \bar{S}}\Osr (E_{\bar{S}',\tilde{S}})\oplus \Osr(E_{\infty,\tilde{S}}).
\end{equation}
Thus, we calculate $ \td_{(1+y)*}\left[ \mathcal{E}_{S,\mathbf{e}(-\alpha), q}\right]$ by combining equations (\ref{T1}), (\ref{T2}), and (\ref{T3}). With the calculation of $n_{S,\alpha}$ we can calculate each $T(S)$ and get  
$M_y(\Asr)$ by equation (\ref{Tsum}).

\subsection{Reduced line arrangements in $\mathbb{P}^2$}\label{line}

In the case of $n=2$ we only need to consider $j=1$ case since the singular locus $\Sigma$ is a finite set of points. In other words $P\in \mathcal{S}_1^{(2)}$ for each point $P\in \Sigma$ and $T(P)$ is determined only by the Hodge spectrum $Sp_P(t)$ as remarked above. Since the spectrum is combinatorial, we can use a local equation which defines a combinatorial equivalent arrangement in a neighborhood of $P$ such that proposition \ref{thom} is applicable. Since there are $m_P$ lines passing through $P$, a such equation is $f_P = x^{m_P} + y^{m_P}$. The Hodge spectrum of $P$ is  
\begin{equation*}
\begin{split}
&Sp_{f_P,0}(t)=Sp_{x^{m_P},0}Sp_{y^{m_P},0}=\left( \frac{t^{1/m_P}-t}{1-t^{1/m_P}} \right)^2=\left( \sum_{k=1}^{m_P-1} t^{\frac{k}{m_P}} \right)^2\\
&=t^{2/m_P}+2t^{3/m_P}+3 t^{4/m_P}+\cdots+(m_P-1) t^{m_P/m_P}+\cdots+2 t^{(2 m_P-3)/m_P}+ t^{(2m_P-2)/m_P}
\end{split}
\end{equation*}

Equivalently, we have
\begin{prop}\label{n2}
Assume that $S$ is a singular point in a reduced line arrangement in $\mathbb{P}^2$. For $k\in\{1,\cdots, m_P\}$ we have
\begin{equation*}
\begin{split}
n_{P,\frac{k}{m_P}}&=k-1 \text{ and}\\
n_{P,1+\frac{k}{m_P}}&=m_P-k-1+\delta_{k,m_P}
\end{split}
\end{equation*}
where $\delta_{k,m_P}=1$ if $k=m_P$ and $0$ otherwise.
\end{prop}

We have $q=0$, $n=2$, and $(\pi_{\tilde{S}})_* \td_{(1+y)*}\left[ \mathcal{E}_{P,\alpha, q}\right]=[P]$ since $P$ is a point. Moreover,
\begin{equation*}
\begin{split}
\sum_{\alpha,q}(-1)^{q+n-1}n_{P,\alpha}(-y)^{\lfloor n-\alpha \rfloor +q }&=\sum_{\alpha}(-1)n_{P,\alpha}(-y)^{\lfloor 2-\alpha \rfloor}\\
&=\left(\sum_{0<\alpha \leq 1}n_{P,\alpha}\right)y-\left(\sum_{1<\alpha \leq 2}n_{P,\alpha}\right) \\
&=\left( \sum_{k=1}^{m_P} n_{S,\frac{k}{m_P}}\right)y-\left(\sum_{k=1}^{m_P} n_{P,1+\frac{k}{m_P}}\right) \\
&=\binom{m_P}{2}y-\binom{m_P-1}{2}
\end{split}
\end{equation*}
from Proposition \ref{n2}.
It follows that 
\begin{equation*}
T(S)=\left(\binom{m_P}{2}y-\binom{m_P-1}{2}\right)[P].
\end{equation*}

Thus, we have the following result.

\begin{theorem}\label{HMP2}
Assume $\Asr$ is a reduced hyperplane arrangement with $m$ hyperplanes in $\mathbb{P}^2$. Let $P$ be a singular point contained in $m_P$ lines. Then

\begin{equation*}
M_y(\Asr)=\sum_{P}  \left(\binom{m_P}{2}y-\binom{m_P-1}{2} \right) [P].
\end{equation*}

\end{theorem}

\subsection {Reduced plane arrangements in $\mathbb{P}^3$}\label{plane}

To calculation $M_y(\Asr)$ we need to consider $j=1$ and $j=2$ both. In the case of $j=1$ we also have two kinds of strata, $\dim S=1$ and $\dim S=0$. In this section we use the notation $S$ only for the stratum $S\in \mathcal{S}_1^{(2)}$. We denote $S_{\infty}=S \cap H'_1$ for a stratum in $\mathcal{S}_2^{(2)}$ since there is a one-to-one correspondence between the sets $\mathcal{S}_2^{(2)}$ and $\mathcal{S}_1^{(2)}$. We denote each element in $\mathcal{S}_1^{(3)}$ by $P$. Thus we have the decomposition,
\begin{equation*}
M_y(\Asr) = \sum_{S\in \mathcal{S}_1^{(2)}} T(S)+\sum_{S_\infty\in \mathcal{S}_2^{(2)}} T(S_\infty)+\sum_{P\in \mathcal{S}_1^{(3)}} T(P).
\end{equation*}
 Here $T(S_\infty)$ and $T(P)$ are calculated only by their spectrum since they are points.

We first calculate the spectrum for each stratum. We use proposition \ref{thom} to calculate the Hodge spectrum of $S$ and $S_\infty$. Since the spectrum of hyperplane arrangements depends only on combinatorial data we can choose the defining polynomial of $S$ as $g_S(w_1,w_2,w_3)=g_{1}(w_1)+g_{2}(w_2, w_3)$ with $g_1(w_1)=0$ and $g_2(w_3)=w_2^{m_S}+w_3^{m_S}$ . To calculate the spectral multiplicity of $S_\infty$ we use the polynomial  $g_{S_{\infty}} = g_S (0, w_2, w_3) + w_1^m = w_1^{m} + w_2^{m_S} + w_3^{m_S}$ by the choice of $f_j$ (item (3) before theorem 3.1).  The formulas for the Hodge spectrum of $S$ and $S_{\infty}$ are given as follows.

\begin{equation*}
Sp_{S_\infty}(t)=\left( \frac{t^{1/m_S}-t}{1-t^{1/m_S}} \right)^2\left( \frac{t^{1/m}-t}{1-t^{1/m}} \right) \text { and }
Sp_S(t)=-t\left( \frac{t^{1/m_S}-t}{1-t^{1/m_S}} \right)^2 .
\end{equation*}
Thus, $n_{S_\infty,\frac{k}{m_S\cdot m}}$ is the coefficient of $t^\frac{k}{m_S \cdot m}$  in $Sp_{S_\infty}(t)$ and $n_{S,\frac{k}{m_S}}$ is the coefficient of $t^\frac{k}{m_S}$ in $Sp_S(t)$. 

We also need a formula for the Hodge spectrum of $P$. This is found in \cite{MR2640043} theorem 3 and \cite{MR3344208} corollary 1.2. 

\begin{theorem}\label{n3}\cite{MR2640043}
Assume  $P$ is a dimension $0$ edge in a reduced hyperplane arrangement $\Asr\subset \mathbb{P}^3$. Let $S$ be a dimension $1$ edge in $\Asr$. The multiplicities $m_{P}$ and $m_{S}$ are given by the number of hyperplanes passing through $P$ and $S$ respectively. Then we have the following formula for $k\in \{1,\cdots, m_{P} \}$:
\begin{equation*}
\begin{split}
n_{P,\frac{k}{m_{P}}}&=\binom{k-1}{2}-\sum_{P\subset \bar{S}}\binom{\lceil k m_S/m_{P}\rceil -1}{2},\\
n_{P,1+\frac{k}{m_{P}}}&=(k-1)(m_{P}-k-1)-\sum_{P\subset \bar{S}}(\lceil k m_S/m_{P}\rceil -1)(m_S-\lceil k m_S/m_{P}\rceil), \text{and}\\
n_{P,2+\frac{k}{m_{P}}}&=\binom{m_{P}-k-1}{2}-\sum_{P\subset \bar{S}}\binom{m_S-\lceil k m_S/m_{P}\rceil}{2}-\delta_{k,m_{P}}
\end{split}
\end{equation*}
where $\delta_{k,m_{P}}=1$ if $k=m_{P}$ and $0$ otherwise. Otherwise $n_{P,\alpha}=0$.  
\end{theorem}

The formula in \cite{MR2640043} uses only non-normal crossing $\bar{S}$ but there is no harm to choose all the set of $\bar{S}$ which include $P$.


Now it is enough to calculate $T(S)$ for $S\in \mathcal{S}_1^{(1)}$. In this case $\pi_{\tilde{S}}:\tilde{S} \to \bar{S}=\mathbb{P}^1$ is the identity map and $\bar{S}\setminus S$ is a finite set $D_{\tilde{S}}=\{E|E=P\in \bar{S}\text{ or }E=S_{\infty} \}$. 
From combinatorial equation (\ref{local}) we get
\begin{equation*}\
\begin{split}
\mathcal{L}_{\tilde{S},\mathbf{e}(\frac{k}{m_S})}&=\Osr_{\mathbb{P}^1}\left(\sum_{P\subsetneq \bar{S}}\left(\lceil  k m_{P,S}/m_S  \rceil -1\right) +\left( \lceil k m_{\infty,S}/m_S \rceil -1\right)  \right) \\
&=\Osr_{\mathbb{P}^1}\left(\sum_{P\subsetneq \bar{S}}\left(\lceil  k \left( m_{P}/m_S-1\right)  \rceil -1\right) +\left( \lceil k \left(1-m /m_S \right) \rceil -1\right)  \right)
\end{split}
\end{equation*}
since  $[E_{\bar{S}',\tilde{S}}]= [E_{\infty,\tilde{S}}]$ in $\tilde{S}=\bar{S}=\mathbb{P}^1$.  Notice that $m_S$ is the number of hyperplanes passing through $S$ and $m$ is the number of hyperplanes in $\Asr$ since we consider only reduced cases.

To see $\Omega^1_{\tilde{S}}\left( \log D_{\tilde{S}} \right)$ we look at the exact sequence
\begin{equation*}
\begin{split}
0 \to \Omega^1_{\tilde{S}} \to \Omega^1_{\tilde{S}}\left( \log D_{\tilde{S}} \right) \to \bigoplus_{E\in D_{\tilde{S}}}\Osr_{\tilde{S}}(E)\to 0
\end{split}
\end{equation*}
where $E$ is a point in $D_{\tilde{S}}$. Notice that $E$ is $P\subsetneq \bar{S}$ or $S_{\infty}$. Thus we get the Chern class
\begin{equation*}
\begin{split}
c\left(\Omega^1_{\tilde{S}}\left( \log D_{\tilde{S}} \right) \right)=(1-2h)(1+h)^{|D_{\tilde{S}}|}=1+(|D_{\tilde{S}}|-2)h \in \mathbf{H}^{\bullet}(\bar{S}),
\end{split}
\end{equation*}
where $h = c_1(\mathscr{O}_{\mathbb{P}^1}(1))$.
We can rewrite $|D_{\tilde{S}}|=\left( \sum_{P\subsetneq \bar{S}} 1\right)+1$. Putting 
\begin{equation*}
\begin{split}
l_{S,k}:=\sum_{P\subsetneq \bar{S}}\lceil  k \left( m_{P}/m_S-1\right)  \rceil  +\lceil k \left(1-m /m_S \right) \rceil.
\end{split}
\end{equation*}
we also get a simplified equation
\begin{equation*}
\begin{split}
\mathcal{L}_{\tilde{S},\lambda}=\mathcal{L}_{\tilde{S},\mathbf{e}(\alpha)}=\mathcal{L}_{\tilde{S},\mathbf{e}(-\frac{k}{m_S})}=\Osr_{\mathbb{P}^1}\left(l_{S,k}-|D_{\tilde{S}}| \right)
\end{split}
\end{equation*}

By the definition of $\td_{(1+y)*}$ we get 
\begin{equation}\label{td1}
\begin{split}
\td_{(1+y)*}[\mathcal{E}_{S,\mathbf{e}(-\frac{k}{m_S}), 1}]=\frac{1}{1+y}[\bar{S}]+\left( l_{S,k} -1\right)[pt]
\end{split}
\end{equation}

From $\Omega^0_{\tilde{S}}\left( \log D_{\tilde{S}} \right)=\Osr_{\tilde{S}}$ we also get 

\begin{equation}\label{td0}
\begin{split}
\td_{(1+y)*}[\mathcal{E}_{S,\mathbf{e}(-\frac{k}{m_S}), 0}]=\frac{1}{1+y}[\bar{S}]+\left( l_{S,k}-\sum_{P\subsetneq \bar{S}} 1 \right)[pt].
\end{split}
\end{equation}

Plugging the equations (\ref{td1}) and (\ref{td0}) into the equation (\ref{TS}) we get
\begin{equation}\label{TS1}
\begin{split}
T(S)&=\sum_{k=1}^{3m_S-1}n_{S,\frac{k}{m_S}} \left( [\bar{S}]+\left(\left( l_{S,k}-1 \right) y+ \left(l_{S,k}-\sum_{P\subsetneq \bar{S}} 1 \right) \right) [pt] \right)(-y)^{\lfloor 3-\frac{k}{m_S} \rfloor }\\
 &=\sum_{k=1}^{3m_S-1}n_{S,\frac{k}{m_S}}\left( [\bar{S}]+ \left(l_{S,k}(y+1)-\left( y+\sum_{P\subsetneq \bar{S}} 1  \right) \right) [pt] \right)(-y)^{\lfloor 3-\frac{k}{m_S} \rfloor }
\end{split}
\end{equation}
 for the case $S\in \mathcal{S}_1^{(1)}$. We can calculate the coefficient polynomial of $[\bar{S}]$ as \S \ref{line}.
 \begin{equation*} 
 \sum_{k=1}^{3m_S-1}n_{S,\frac{k}{m_S}} (-y)^{\lfloor 3-\frac{k}{m_S} \rfloor }=\binom{m_S}{2}y-\binom{m_S-1}{2} 
 \end{equation*}

We summarize the result as the following theorem.
\begin{theorem}\label{HMP3}
Assume $\Asr$ is a reduced hyperplane arrangement with $m$ hyperplanes in $\mathbb{P}^3$. Let $P$ and $\bar{S}$ be the lattices of dimension $0$ with the multiplicity $m_{P}$ and dimension $1$ with multiplicity $m_{S}$ respectively. Then
\begin{equation*}
\begin{split}
M_y(\Asr)=&\sum_{\bar{S}} \left(\binom{m_S}{2}y-\binom{m_S-1}{2}\right)[\bar{S}]\\
&+\left( \sum_{P}\left( \sum_{k=1}^{3m_{P} -1}n_{S,\frac{k}{m_{P}}} (-y)^{\lfloor 3-\frac{k}{m_{P}} \rfloor }\right)+ \sum_{\bar{S}}\left( \sum_{k=1}^{3m_S \cdot m -1}n^\infty_{S,\frac{k}{m_S \cdot m }} (-y)^{\lfloor 3-\frac{k}{m_S \cdot m } \rfloor } \right. \right.\\
&\left. \left. +\sum_{k=1}^{3m_S-1}n_{S,\frac{k}{m_S}} (-y)^{\lfloor 3-\frac{k}{m_S} \rfloor }\left(l_{S,k}(y+1)-\left( y+ \sum_{P\subsetneq \bar{S}} 1  \right) \right)\right) \right)[pt]
\end{split}
\end{equation*}
where $l_{S,k}=\sum_{P\subsetneq \bar{S}}\lceil  k \left( m_{P}/m_S-1\right)  \rceil  +\lceil k \left(1-m /m_S \right) \rceil$, $n^\infty_{S,\frac{k}{m_S\cdot m}}$ is the coefficient of $t^\frac{k}{m_S \cdot m}$  in $Sp^\infty_{S}(t):=\left( \frac{t^{1/m_S}-t}{1-t^{1/m_S}} \right)^2\left( \frac{t^{1/m}-t}{1-t^{1/m}} \right)$, $n_{S,\frac{k}{m_S}}$ is the coefficient of $t^\frac{k}{m_S}$ in $Sp_{S}(t):=-t \left( \frac{t^{1/m_S}-t}{1-t^{1/m_S}} \right)^2 $, and $n_{P,\frac{k}{m_{P}}}$ is given in Theorem \ref{n3}. 
\end{theorem}

\begin{example}\label{exP3}
Consider the arrangement $\Asr$ in $\Pbb^3$ defined by the equation $xyz(x+y)=0$. 
It has $m=4$ hyperplanes $H_1:=\{ (x;y;z;w)\in \Pbb^3| x=0 \}$, $H_2:=\{ (x;y;z;w)\in \Pbb^3| y=0 \}$, $H_3:=\{ (x;y;z;w)\in \Pbb^3| x+y=0 \}$, and $H_4:=\{ (x;y;z;w)\in \Pbb^3| z=0 \}$. For $j=1$ we choose the generic hyperplane $H':=\{ (x;y;z;w)\in \Pbb^3| w=0 \}$. The singular locus $\Sigma_1$ have the strata $\mathcal{S}_1=\{ P, S_1, S_2, S_3, S_4 \}$ such that $P=H_1 \cap H_2 \cap H_3 \cap H_4$, $\bar{S_1}=H_1 \cap H_4$, $\bar{S_2}=H_2 \cap H_4$, and $\bar{S_3}=H_3 \cap H_4$, and $\bar{S_4}=H_1 \cap H_2 \cap H_3$. The multiplicities are given by $m_{P}=4$, $m_{S_1}=m_{S_2}=m_{S_3}=2$, and $m_{S_4}=3$. Thus, 
\begin{equation*}\
\begin{split}
l_{S_i,k}&=\begin{cases}\lceil k (4/2-1)\rceil+\lceil k (1-4/2) \rceil=0 &\text{ for } i=1,2,3\\
 \lceil k (4/3-1)\rceil+\lceil k (1-4/3)\rceil=\begin{cases}0& \text{if } k\equiv 0 \pmod{3} \\ 1&\text{otherwise} \end{cases}
  &\text{ for } i=4.\end{cases}\\
 Sp_{S_i}(t)&=\begin{cases}((t^{\frac{1}{2}})^2 (-t)^=-t^2 &\text{ for } i=1,2,3\\
 (t^{\frac{1}{3}}+t^{\frac{2}{3}})^2(-t)=-t^{\frac{5}{3}}-2t^{\frac{6}{3}}-t^{\frac{7}{3}} &\text{ for } i=4.\end{cases}\\
 Sp^\infty_{S_i}(t)&=\begin{cases}(t^{\frac{1}{2}})^2(t^{\frac{1}{4}}+t^{\frac{2}{4}}+t^{\frac{3}{4}}) &\text{ for } i=1,2,3\\
 (t^{\frac{1}{3}}+t^{\frac{2}{3}})^2(t^{\frac{1}{4}}+t^{\frac{2}{4}}+t^{\frac{3}{4}}) &\text{ for } i=4.\end{cases}\\
 Sp_{P}(t)&=(-2t^2+3t^3)(-t)^{-1}=2t-3t^2
\end{split}
\end{equation*}

Finally we get the result
\begin{equation*}
M_y(\Asr)=y \sum_{i=1}^3 [S_i]+(3y-1) [S_4]+(-2y^2-21y+1)[pt],
\end{equation*}
where $[pt] \in \mathbf{H}_{0}(\Sigma)$  and $[S_1],[S_2],[S_3],[S_4] \in \mathbf{H}_{1}(\Sigma)$ are the generators of $\mathbf{H}_{\bullet}(\Sigma)$.

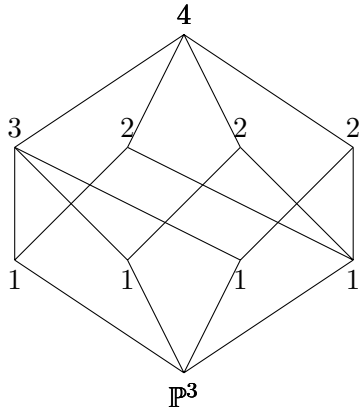
\begin{figure}[!h]
\begin{tikzpicture}[scale=1.5] 
   \draw (1,1) -- (1,2) (1,1) -- (2,2);
   \draw (2,1) -- (1,2) (2,1) -- (3,2);
   \draw (3,1) -- (1,2) (3,1) -- (4,2);
   \draw (4,1) -- (3,2) (4,1) -- (4,2) (4,1) -- (2,2);      
   \draw (2.5,0) ;
   \draw \foreach \x in {1,...,4}{
       (2.5,0) node[below]{$\mathbb{P}^3$} -- (\x,1) node[below]{1} 
       (\x,2) -- (2.5,3) node[above]{4}       
       }; 
   \draw (1,2) node[above]{3}; 
   \draw \foreach \y in {2,3,4}{
       (\y,2) node[above]{2}
       };
\end{tikzpicture}

\caption{The intersection lattice of the affine arrangement $xyz(x+y)=0$ and the multiplicities of edges.}
\end{figure}

\end{example}

\section{the combinatorial formula for Hirzebruch-Milnor classes via K-theoretic computations}\label{K-theoretic-hm}

\subsection{K-theoretic algorithm for Hirzebruch-Milnor classes of hyperplane arrangements}\label{K-algorithm}
Let $\mathscr{A}$ be a hyperplane arrangement in $\mathbb{P}^n$, and let $\Sigma$ be its singular locus. The Hirzebruch-Milnor class of the hyperplane arrangement $M_y(\mathscr{A})$ is a class living in $\mathbf{H}_{\bullet}(\Sigma)[y]$. Applying theorem \ref{hmdiff} we have

\begin{equation*}
(i_{\mathscr{A},\mathbb{P}^n})_*\Big(T^{vir}_{y*}(\mathscr{A}) - T_{y*}(\mathscr{A})\Big) = (i_{\Sigma,\mathbb{P}^n})_*M_y(\mathscr{A}) \in \mathbf{H}_{\bullet}(\mathbb{P}^n)[y].
\end{equation*}

By our discussions in \S \ref{general}, this equation can be expressed as 

\begin{equation}\label{hm2}
(i_{\Sigma,\Pbb^n})_* M_y(\Asr) = \textup{td}_{(1+y)_*}\Bigg((i_{\Asr,\Pbb^n})_*\Big(\textup{DR}^{vir}_{y}[\Asr] - mC_y([\Asr \to \Asr])\Big)\Bigg).
\end{equation}

Let us fix some blanket notations for the rest of this paper. We will write $1$ and $t$ for the classes of $\mathscr{O}_{\Pbb^n}$ and $\mathscr{O}_{\Pbb^n}(-1)$ in $K_0(\Pbb^n)$ respectively. We let $h = c_1(\mathscr{O}_{\mathbb{P}^n}(1))$. So $\textup{ch}(t) = e^{-h}$.

\begin{prop}\label{vir}
Let $m$ be the degree of the hyperplane arrangements $\Asr$. The class $(i_{\Asr,\Pbb^n})_*\textup{DR}^{vir}_{y}[\Asr]$ is given by:
\begin{displaymath}
\frac{(1+ty)^{n+1}(1-t^m)}{(1+y)(1+t^my)}
\end{displaymath}
\end{prop}

\begin{proof}
This follows immediately from equation \eqref{virhyper}, $\mathscr{O}_{\mathbb{P}^n}(\mathscr{A}) \cong \mathscr{O}_{\mathbb{P}^n}(m)$, and 
\begin{equation*}
\textup{DR}_y[\mathbb{P}^n] = \frac{(1+ty)^{n+1}}{1+t}.
\end{equation*}
The last equation is the content of lemma 5.1(i) in \cite{Kdefect}.
\end{proof}

To give a combinatorial expression for $(i_{\Asr,\Pbb^n})_*mC_y([\Asr \to \Asr])$, we recall the following result.

\begin{theorem}(\cite{Kdefect} theorem 5.2)
Let $\mathscr{A}$ be a hyperplane arrangement in $\mathbb{P}^n$. Denote by $\hat{\mathscr{A}}$ the affine cone of $\mathscr{A}$ in $\mathbb{A}^{n+1}$ and let $\chi_{\hat{\mathscr{A}}}$ be the characteristic polynomial of $\hat{\mathscr{A}}$. We have
\begin{equation*}
mC_y([(\mathbb{P}^n \setminus \mathscr{A}) \to \mathbb{P}^n]) = \frac{(1-t)^{n+1}}{1+y} \chi_{\hat{\mathscr{A}}}\Big(\frac{1+ty}{1-t}\Big)
\end{equation*}
\end{theorem}

It follows that we have 

\begin{equation}\label{mc}
\begin{split}
(i_{\Asr,\Pbb^n})_*mC_y([\Asr \to \Asr]) =& mC_y([\Asr \to \Pbb^n]) \\
                                                               =& mC_y([\Pbb^n \to \Pbb^n]) - mC_y([\mathbb({P}^n \setminus \mathscr{A}) \to \Pbb^n]) \\
                                                               =& \frac{(1+ty)^{n+1}}{1+y} - \frac{(1-t)^{n+1}}{1+y} \chi_{\hat{\mathscr{A}}}\Big(\frac{1+ty}{1-t}\Big) \\
                                                               =& \frac{(1-t)^{n+1}}{1+y} \Big(w^{n+1} - \chi_{\hat{\mathscr{A}}}(w)\Big),
\end{split}
\end{equation}
where we let $w = \frac{1+ty}{1-t}$. We will use the notation $w = \frac{1+ty}{1-t}$ until the end of the paper.

Combining equation \eqref{hm2}, equation \eqref{mc} and proposition \ref{vir}, we have shown the following theorem. 
\begin{theorem}
 $(i_{\Sigma,\Pbb^n})_* M_y(\Asr)$ is a combinatorial object.
\end{theorem}

\begin{remark}\label{info-preserve}
The combinatorial nature of $M_y(\Asr)$ was proven in \cite{Maxim2016} proposition 4.6 and 4.7. Our approach only proves the weaker result that  $(i_{\Sigma,\Pbb^n})_* M_y(\Asr)$ is combinatorial, but is much more elementary. On the other hand, according to proposition \ref{hmlgsigma}, the map $(i_{\Sigma,\Pbb^n})_*: \mathbf{H}_k(\Sigma) \to \mathbf{H}_k(\mathbb{P}^n)$ is an isomorphism except for $k = \dim \Sigma$. So not much information about $M_y(\mathscr{A})$ is lost when considering its image in $\mathbf{H}_{\bullet}(\mathbb{P}^n)[y]$. 
\end{remark}

\begin{remark}
To calculate the virtual Hirzebruch class of a hyperplane arrangement, one takes the power series expansion of the rational function (in the variable $y$) in proposition \ref{vir} first, and then apply $\textup{td}_{(1+y)*}$ to the expansion. This method of computing the virtual Hirzebruch class can be assigned to a computer, but it is hardly achievable by hand. We will approach the computation in the case of line arrangement and plane arrangements ($n=2,3$) by a different method. 
\end{remark}

To proceed the computation of $T_{y*}(\mathscr{A})$, we see that by equation \eqref{mc}, we need to compute $\td_{(1+y)*}\Big(\frac{(1-t)^{n+1}}{1+y}w^m\Big)$ for those integers $m$ such that $0 \leq m \leq n$, since $\chi_{\hat{\mathscr{A}}}(x)$ is a monic polynomial of degree $n+1$.

\begin{lemma}
Let $0 \leq m \leq n$. We have

\begin{equation*}
\td_{(1+y)*}\Big(\frac{(1-t)^{n+1}}{1+y}\Big) = 0 \ \text{and}
\end{equation*}

\begin{equation*}
\td_{(1+y)*}\Big(\frac{(1-t)^{n+1}}{1+y}w^{m+1}\Big) = (i_{\mathbb{P}^m,\mathbb{P}^n})_*T_{y_*}(\mathbb{P}^m).
\end{equation*}
\end{lemma}

\begin{proof}
The first equality follows from the Koszul relation $(1-t)^{n+1} = 0$. The second equality follows immediately from lemma 5.1(iii) in \cite{Kdefect}, the naturality of $\td_{(1+y)*}$, and the definition of Hirzebruch class.
\end{proof}

\begin{corol}\label{Hirzarr}
Let $\mathscr{A}$ be a reduced hyperplane arrangement in $\mathbb{P}^n$, and let $\chi_{\hat{\mathscr{A}}}(x) = \displaystyle{\sum_{i=0}^{n+1}} (-1)^ic_ix^{n+1-i}$ be as above ($c_0=1$). Then

\begin{equation*}
(i_{\Asr,\mathbb{P}^n})_*T_{y*}(\mathscr{A}) = \sum_{i=1}^{n}(-1)^{i+1} c_{i} T_{y*}([\mathbb{P}^{n-i}]) \in \mathbf{H}_{\bullet}(\mathbb{P}^n)[y]
\end{equation*}

\end{corol}

\begin{remark}
If follows easily from the definition of the M\"{o}bius function of the intersection lattice, that $\chi_{\hat{\mathscr{A}}}(1) = 0$, so that we can introduce the polynomial $\frac{\chi_{\hat{\mathscr{A}}}(x)}{x-1} = \sum_{i=0}^n (-1)^i\mu^i x^{n-i}$. The $c$'s and $\mu$'s are related by $c_i = \mu^i + \mu^{i-1}$ for $i = 0,\ldots,n+1$, where we set $\mu^{-1} = \mu^{n+1} = 0$. The $\mu$'s assume a good number of interesting interpretations, either algebraic or topological. For instance, they can be interpreted as the mixed multiplicities of a certain ideal of partial derivatives, as the numbers of cells in a CW-model of $\mathbb{P}^n\setminus \mathscr{A}$, etc.. See \cite{MR2904577} for a detailed study of these connections.
\end{remark}

Let us calculate $T_{y_*}(\mathbb{P}^m)$. We remind the readers that the result was already known as a special case of the Hirzebruch classes of simplicial toric varieties computed in \S 5.3 of \cite{MR3417881}, but the method we use here, in a similar vein with the style of this paper, is more straightforward.

We first calculate the unnormalised class $\tilde{T}_{y_*}(\mathbb{P}^m)$. By the definition of $\tilde{T}_{y_*}$ and Lemma 5.1 (i) in \cite{Kdefect}, we have

\begin{align*}
\tilde{T}_{y_*}(\mathbb{P}^m) =& \td_{*}\Big(\frac{(1+ty)^{m+1}}{1+y}\Big)\cdot \Big(\frac{h}{1-e^{-h}}\Big)^{m+1} \cap [\mathbb{P}^m] \\
                                      =&\frac{1}{1+y}\Big(\frac{1+e^{-h}y}{1-e^{-h}}\Big)^{m+1} h^{m+1} \cap [\mathbb{P}^m] \\
                                      =& \frac{1}{1+y}(-y+\frac{y+1}{1-e^{-h}})^{m+1} h^{m+1} \cap [\mathbb{P}^m] \\
                                      =& \sum_{i=0}^m \binom{m+1}{i} (-y)^i (1+y)^{m-i} \frac{h^{m+1}}{(1-e^{-h})^{m+1-i}} \cap [\mathbb{P}^m]
\end{align*} 

The coefficient in front of $[\mathbb{P}^{m-j}]$ is the coefficient of $h^j$ in the Laurent series expansion of 
\begin{equation*}
\sum_{i=0}^m \binom{m+1}{i} (-y)^i (1+y)^{m-i} \frac{h^{m+1}}{(1-e^{-h})^{m+1-i}}.
\end{equation*} 
Therefore, $\tilde{T}_{y_*}(\mathbb{P}^m)$ can be written equivalently as 
\begin{equation*}
\tilde{T}_{y_*}(\mathbb{P}^m) = \sum_{j=0}^m\sum_{i=0}^m \binom{m+1}{i} (-y)^i (1+y)^{m-i} \Res\Big(\frac{h^{m-j}}{(1-e^{-h})^{m+1-i}}\Big)[\mathbb{P}^{m-j}]
\end{equation*}

\begin{lemma}
\begin{equation*}
\Res\Big(\frac{h^{m-j}}{(1-e^{-h})^{m+1-i}}\Big) = \Res\Big((\ln(1+x))^{m-j}(1+\frac{1}{x})^{m-i}\frac{1}{x}\Big).
\end{equation*}
\end{lemma}

\begin{proof}
\begin{align*}
\oint \frac{z^{m-j}dz}{(1-e^{-z})^{m+1-i}} &= \oint \frac{z^{m-j}(e^z)^{m-i}de^z}{(e^z-1)^{m+1-i}} \\
                                                               &= \oint \frac{(\ln(1+x))^{m-j}(x+1)^{m-i}dx}{x^{m+1-i}}
\end{align*}
\end{proof}

From now on, we let $a_{m,i,j} = \Res\Big((\ln(1+x))^{j}(1+\frac{1}{x})^{m-i}\frac{1}{x}\Big)$, which is also the constant term in the Laurent series expansion of $(\ln(1+x))^{j}(1+\frac{1}{x})^{m-i}$. 

\begin{prop}\label{HirzPm}
\begin{equation*}
T_{y*}(\mathbb{P}^m) =  \sum_{j=0}^m\sum_{i=0}^m a_{m,i,j}\binom{m+1}{i} (-y)^i (1+y)^{m-i-j} [\mathbb{P}^j]
\end{equation*}
\end{prop}

\begin{proof}
This follows from the computation of $\tilde{T}_{y_*}(\mathbb{P}^m)$ and the normalisation rule \eqref{normalisation}.
\end{proof}

The following properties of $a_{m,i,j}$ are easy to verify.

\begin{lemma}\label{amij}
\leavevmode
\begin{enumerate}[(i)]
\item $a_{m,i,j} = a_{m-i,0,j}$.
\item $a_{m,i,j} = 0$ if $i+j > m$, and $a_{m,i,j} = 1$ if $i+j = m$.
\item $a_{m,i,0} =1 $. 
\item In general, we have
\[
a_{m,i,j} = \sum_{k=1}^{m-i} \sum_{\substack{i_1,\ldots,i_j \\ i_1 + \ldots + i_j = k}}\binom{m-i}{k} \frac{(-1)^k}{i_1\ldots i_j}
\]
\end{enumerate}
\end{lemma}

\begin{remark}
We can immediately recover some well-known formulas with the help of proposition \ref{HirzPm} and lemma \ref{amij}.
For example, we can let $y=-1$ in proposition \ref{HirzPm} and use lemma \ref{amij} (ii). This gives us
\begin{equation*}
c(T\mathbb{P}^m) \cap [\mathbb{P}^m] = T_{-1*}(\mathbb{P}^m) = \sum_{i=0}^m \binom{m+1}{i}[\mathbb{P}^{m-i}] = (1+h)^{m+1}\cap [\mathbb{P}^m].
\end{equation*}

Using lemma \ref{amij} (iii), we get
\begin{equation*}
\chi_y(\mathbb{P}^m) = \deg (T_{y*}(\mathbb{P}^m)) = \sum_{i=0}^m \binom{m+1}{i} (-y)^i (1+y)^{m-i} = \sum_{i=0}^m (-y)^i.
\end{equation*}
\end{remark}

We will simplify the equation \eqref{hm2} in the case of line arrangement and plane arrangements ($n=2,3$). For this purpose, the following computational results will be particularly useful.
\begin{example}\label{calc012}

\begin{equation*}
T_{y*}(\mathbb{P}^2) = [\mathbb{P}^2] + \frac{3}{2}(1-y)[\mathbb{P}^1] + (1-y+y^2)[pt],
\end{equation*}

\begin{equation*}
T_{y*}(\mathbb{P}^1) = [\mathbb{P}^1] + (1-y)[pt],
\end{equation*}

\begin{equation*}
T_{y*}(pt) = [pt].
\end{equation*}
\end{example}

\subsection{Revisiting reduced line arrangements in $\mathbb{P}^2$}\label{Kline}

We focus once again on the case that $\Asr$ is a reduced line arrangement in $\Pbb^2$. This case is conceptually easy to deal with because the singularities of $\Asr$ are isolated. The Hirzebruch-Milnor class $M_y(\Asr)$ should be the sum of the local contributions. In other words, if we denote the singular points of $\Asr$ by $P_1, \ldots, P_k$, and denote the line arrangement consisting only the lines passing through $P_i$ by $\Asr_i$, then  
\begin{equation}\label{hm3}
M_y(\Asr) = \sum M_y(\Asr_{P_i}) \in \mathbf{H}_{\bullet}(\Sigma)[y]
\end{equation}
Indeed, this follows easily from \cite{MR3053711} corollary 2. 

Unfortunately, the method developed in \S \ref{K-algorithm} does not allow us to prove equation \eqref{hm3}. We can only prove the weaker version that \eqref{hm3} holds in $\mathbf{H}_{\bullet}(\Pbb^2)[y]$. At any rate, this is what we will do next. After that, we will give a formula for $(i_{\mathscr{A}_{P_i},\mathbb{P}^2})_*(M_y(\mathscr{A}_{P_i}))$. The step of checking \eqref{hm3} may seem pedantic, but we think it conceptually natural, and is also good for the readers to get familiar with the procedure of computing Hirzebruch-Milnor classes by equation \eqref{hm2}. 

Let us first fix the combinatorial data of our line arrangement and draw some easy conclusions:
\begin{enumerate}
\item There are $m$ lines in the arrangement, denoted by $L_1, \ldots, L_m$. 
\item There are $k$ singular points of $\Asr$, denoted by $P_1, \ldots, P_k$.
\item There are $m_i$ lines passing through $P_i$.
\item The M\"{o}bius function of the arrangement $\hat\Asr$ is $\mu$. Its value on $\hat{L_i}$ and $\hat{P_i}$ are simply denoted by $\mu(L_i)$ and $\mu(P_i)$. Since $\hat\Asr$ is a central arrangement, the origin of $\Abb^3$ is in the intersection lattice. We denote the value of $\mu$ on the origin of $\Abb^3$ by $\mu(1)$, in accordance with the convention of the theory of hyperplane arrangements. 
\item $\mu(L_i) = -1$ and $\mu(P_i) = m_i - 1$ for all $i$.
\item Therefore, $\chi_{\hat\Asr}(x) = x^3 - mx^2 + \sum (m_i-1)x + \mu(1)$.
\end{enumerate}
 
From the last observation, Corollary \ref{Hirzarr}, and Example \ref{calc012}, we get
 \begin{align*}
 (i_{\Asr,\Pbb^2})_*T_{y*}(\Asr) =& mT_{y*}(\mathbb{P}^1) -\sum (m_i - 1)T_{y*}(pt) \\
                                                  =& m[\mathbb{P}^1] + \Big(m(1-y) - \sum (m_i - 1)\Big)[pt].
 \end{align*}
 
Applying this result to the arrangement $\Asr_{P_j}$, we get
\begin{equation*}
(i_{\Asr_{P_j},\Pbb^2})_*T_{y*}(\Asr_{P_j}) = m_i[\mathbb{P}^1] + \Big(m_i(1-y) - (m_i - 1)\Big)[pt].
\end{equation*}

If we express $(i_{\Asr,\Pbb^2})_*T_{y*}(\Asr)$ in terms of a summation of local terms and a defect term, we have
\begin{equation}\label{homdecom}
(i_{\Asr,\Pbb^2})_*T_{y*}(\Asr) = \sum (i_{\Asr_{P_j},\Pbb^2})_*T_{y*}(\Asr_{P_j}) + (m-\sum m_j) \Big([\Pbb^1] + (1-y)[pt]\Big) .
\end{equation}

Next, we have to compute the virtual Hirzebruch class 
\begin{displaymath}
(i_{\Asr,\Pbb^2})_*T^{vir}_{y*}(\Asr) = \textup{td}_{(1+y)*}(i_{\Asr,\Pbb^n})_*\textup{DR}^{vir}_{y}[\Asr].
\end{displaymath}
Of course we can use proposition \ref{vir} and calculate the power series expansion of $(i_{\Asr,\Pbb^n})_*\textup{DR}^{vir}_{y}[\Asr]$ by brutal force, but in the present context, there is a much easier way to carry out the computation. Note that the nearby curves to $\Asr$ defined by global sections of $\Osr(m)$ are nonsingular curves of degree $m$. Let $C$ be any of the nearby curve. By our discussion about the virtual Hirzebruch class in \S \ref{general}, we have
\begin{equation*}
\begin{split}
\textup{td}_{(1+y)*}(i_{\Asr,\Pbb^n})_*\textup{DR}^{vir}_{y}[\Asr] =& (i_{\Asr,\Pbb^n})_*\textup{td}_{(1+y)*}\textup{DR}^{vir}_{y}[\Asr]  \\
                                                                                                      =& (i_{C,\Pbb^n})_*\textup{td}_{(1+y)*}\textup{DR}_{y}[C] \\
                                                                                                      =& (i_{C,\Pbb^n})_*\textup{td}_{(1+y)*}(\Osr_C + \Omega_C \cdot y)
\end{split}
\end{equation*}

To compute $(i_{C,\Pbb^n})_*\textup{td}_{(1+y)*}(\Osr_C + \Omega_C y)$, we must compute $(i_{C,\Pbb^n})_*\textup{td}_*(\Osr_C + \Omega_C y)$ first and then normalise the result. We have
\begin{equation*}
\begin{split}
(i_{C,\Pbb^n})_*\textup{td}_*(\Osr_C + \Omega_C y) =& (i_{C,\Pbb^n})_*\Big(\textup{ch}(\Osr_C + \Omega_C y)\cdot \textup{td}(T_C)\cap [C]\Big) \\
                                                                                    =& (i_{C,\Pbb^n})_*\Big(\big(1+(1+c_1(\Omega_C))y\big)(1+\frac{1}{2}c_1(T_C)) \cap [C]\Big) \\
                                                                                    =&  (i_{C,\Pbb^n})_*\bigg((1+y)[C] + \Big(c_1(\Omega_C)y + \frac{1+y}{2}c_1(T_C)\Big) \cap [C]\bigg) \\
                                                                                    =& m(1+y) [\Pbb^1]+ \frac{m(m-3)}{2}(y-1) [pt] \in \mathbf{H}_{\bullet}(\Pbb^2)[y],
\end{split}
\end{equation*}
where in the last step of the calculation we used that the degree of $C$ is $m$ and the adjunction formula $\textup{deg}(c_1(\Omega_C)\cap [C]) = m(m-3)$. Normalise the result, we get
\begin{equation*}
(i_{\Asr,\Pbb^2})_*T^{vir}_{y*}(\Asr) = m [\Pbb^1] + \frac{m(m-3)}{2}(y-1)[pt] \in \mathbf{H}_{\bullet}(\Pbb^2)[y].
\end{equation*}

We can then formally write $(i_{\Asr,\Pbb^2})_*T^{vir}_{y*}(\Asr)$ as a summation of local terms and a defect term:
\begin{equation}\label{virdecom}
\begin{split}
(i_{\Asr,\Pbb^2})_*T^{vir}_{y*}(\Asr) =& \sum (i_{\Asr_{P_i},\Pbb^2})_*T^{vir}_{y*}(\Asr_{P_i}) \\
                                                   &+ (m-\sum m_i) [\Pbb^1] + \frac{y-1}{2}\Big( m(m-3) - \sum m_i(m_i-3) \Big) [pt]
\end{split}
\end{equation}

Now, compare equation \eqref{homdecom} and equation \eqref{virdecom}. We see that equation \eqref{hm3} is true in $\mathbf{H}_{\bullet}(\Pbb^2)[y]$ if we can show that 
\begin{equation*}
(m-\sum m_i)(1-y) = \frac{y-1}{2}\Big( m(m-3) - \sum m_i (m_i -3)\Big),
\end{equation*}
or equivalently
\begin{equation*}
{m \choose 2} = \sum {m_i \choose 2}.
\end{equation*}
The last equation is evident at once when one considers the ways of choosing two lines out of $d$ lines.

After checking equation \eqref{hm3}, we have legitimately brought the calculation of the Hirzebruch-Milnor class to the local situation, namely, all lines passing through one common point. In this local setup, we have

\begin{equation*}
\begin{split}
(i_{\Asr_{P_i},\Pbb^2})_*M_y(\Asr_{P_i}) =& (i_{\Asr_{P_i},\Pbb^2})_*T^{vir}_{y*}(\Asr_{P_i}) -  (i_{\Asr_{P_i},\Pbb^2})_*T_{y*}(\Asr_{P_i}) \\
                                              =& \Big(m_i [\mathbb{P}^1] + \frac{m_i(m_i-3)}{2}(y-1)[pt] \Big) - \Big(m_i [\mathbb{P}^1] + \big(1-m_iy\big) [pt]\Big) \\
                                              =& \Big( \frac{m_i(m_i-1)}{2}y - \frac{(m_i-1)(m_i-2)}{2} \Big) [pt]
\end{split}
\end{equation*}

\subsection{Revisiting reduced plane arrangements in $\Pbb^3$}

We employ the same notations used in \S \ref{plane} to facilitate the comparison. Let $\mathscr{A}$ be a reduced plane arrangement of degree $m$ in $\Pbb^3$. An 1-dimensional element (an edge) in the intersection lattice $L(\mathscr{A})$ is denoted by $\bar{S}$, and a 0-dimensional element in $L(\mathscr{A})$ is denoted by $P$. Let $m_S$ be the number of planes containing an edge $S$, and let $m_P$ be the number of planes containing a point $P$. With these data, we can write the characteristic polynomial of $\hat{\mathscr{A}}$ as 

\begin{equation*}
\chi_{\hat{\mathscr{A}}}(x) = x^4 - mx^3 + \sum_{\bar{S}} (m_S - 1) x^2 - \sum_P \big(\sum_{P \in \bar{S}}(m_S -1) - m_P + 1\big)x + \mu(1)
\end{equation*}
where the notation $\mu(1) = \chi_{\hat{\mathscr{A}}}(0)$ is chosen to comply with the convention in the theory of hyperplane arrangements.

By Corollary \ref{Hirzarr} and Example \ref{calc012}, we have
\begin{equation}\label{HirzarrP3}
\begin{split}
&(i_{\Asr,\Pbb^3})_*T_{y*}(\mathscr{A}) \\
=& mT_{y*}(\mathbb{P}^2) -  \sum_{\bar{S}} (m_S - 1) T_{y*}(\mathbb{P}^1) +\sum_P \big(\sum_{P \in \bar{S}}(m_S -1) - m_P + 1\big) T_{y*}(pt) \\
=& m[\mathbb{P}^2] + \Big(\frac{3m(1-y)}{2} - \sum_{\bar{S}} (m_S - 1)\Big)[\mathbb{P}^1] \\
  & +\Big(m(1-y+y^2) - \sum_{\bar{S}} (m_S - 1)(1-y) + \sum_P \big(\sum_{P \in \bar{S}}(m_S -1) - m_P + 1\big)\Big)[pt]
\end{split}                                  
\end{equation}

To calculate the virtual Hirzebruch class, our strategy is the same as in the case of line arrangement. Let $X$ be a nonsingular surface of degree $d$, whose defining equation is a small perturbation of the defining equation of $\Asr$. Then we have
\begin{equation*}
\begin{split}
(i_{\Asr,\Pbb^3})_*T^{vir}_y(\Asr) &= (i_{X,\Pbb^3})_*T_y(X) \\
                                                     &=  (i_{X,\Pbb^3})_*\textup{td}_{(1+y)*}(\Osr_X + \Omega^{1}_Xy +\Lambda^2\Omega^{1}_Xy^2).
\end{split}
\end{equation*}

Let $c_i = c_i(\Omega^{1}_X)$, we have
\begin{equation*}
\begin{split}
& \textup{td}_{*}(\Osr_X + \Omega^{1}_Xy +\Lambda^2\Omega^{1}_Xy^2) \\ 
= &\textup{ch}(\Osr_X + \Omega^{1}_Xy +\Lambda^2\Omega^{1}_Xy^2) \textup{td}(T_X) \cap [X] \\
= &\Big(1+(2+c_1 +\frac{1}{2}c_1^2-c_2)y + (1+c_1+\frac{1}{2}c_1^2)y^2\Big) (1-\frac{1}{2}c_1 + \frac{1}{12}(c_1^2+c_2)) \cap [X] \\
= &\Big((1+2y+y^2) + (y+y^2)c_1 + (\frac{y+y^2}{2}c_1^2-yc_2)\Big) (1-\frac{1}{2}c_1 + \frac{1}{12}(c_1^2+c_2))\cap [X] \\
= & \bigg((1+y)^2 + \frac{y^2-1}{2}c_1 + \Big((\frac{(1+y)^2}{12}-y)c_2 + \frac{(1+y)^2}{12}c_1^2\Big)\bigg) \cap [X].
\end{split}
\end{equation*}

Normalise the result, we get
\begin{equation*}
T_y(X) = \bigg(1 + \frac{y-1}{2}c_1 + \Big((\frac{(1+y)^2}{12}-y)c_2 + \frac{(1+y)^2}{12}c_1^2\Big)\bigg) \cap [X]
\end{equation*}

Note that we have the following relations
\begin{equation*}
c_1 = (-4h+mh)\vert_X ,
\end{equation*}
\begin{equation*}
\chi(\Osr_X) = \frac{1}{12}\int_X(c_1^2 + c_2), 
\end{equation*}
\begin{equation*}
\chi(\Osr_X) = \chi(\Osr_{\Pbb^3}) - \chi(\Osr_{\Pbb^3}(-m)) = 1+ \frac{(m-3)(m-2)(m-1))}{6},
\end{equation*}
by the adjunction formula, Noether's formula, and basic properties of Hilbert polynomials respectively. 

Combining these formulas, we get
\begin{equation*}
\chi_e(X) = \int_X c_2 \cap [X] = m^3 -4m^2 +6m
\end{equation*}
and
\begin{equation}\label{HirzvirarrP3}
\begin{split}
&(i_{\Asr,\Pbb^3})_*T^{vir}_y(\Asr) = (i_{X,\Pbb^3})_*T_y(X) \\
=& (i_{X,\Pbb^3})_*\bigg([X] + \frac{y-1}{2}K_X \cdot [X] + \Big((1+y)^2\chi(\Osr_X) -\chi_e(X)y\Big)[pt] \bigg) \\
=& m[\mathbb{P}^2] + \frac{y-1}{2}m(m-4)[\mathbb{P}^1] + \Big((1+\binom{m-1}{3})(1+y)^2 - (m^3-4m^2+6m)y\Big)[pt] 
\end{split}
\end{equation}

\begin{theorem}\label{HMP3'}
Let $\mathscr{A}$ be an arrangement of planes in $\mathbb{P}^3$. With the notations introduced in this subsection, $(i_{\Sigma,\mathbb{P}^3})_*M_y(\mathscr{A})$ is given by
\begin{align*}
(i_{\Sigma,\mathbb{P}^3})_*&M_y(\mathscr{A}) = \sum_{\bar{S}} \Big(\binom{m_S}{2}y - \binom{m_S-1}{2}\Big) [\mathbb{P}^1] \\
                                                                            +& \bigg(\Big(\binom{m-1}{3} - m + 1\Big)y^2 - \Big(4\binom{m}{3} + \sum_{\bar{S}}(m_S-1)\Big)y + \Big(\binom{m-1}{3} + \mu(1)\Big)\bigg)[pt].
\end{align*}
\end{theorem}

\begin{proof}
This follows from equations \eqref{HirzvirarrP3} and \eqref{HirzarrP3}. We leave the tedious verification to our readers. In obtaining the result, we have used the relation $\chi_{\hat{\mathscr{A}}}(1) = 0$ and $\binom{m}{2} = \sum_{\bar{S}}\binom{m_{S}}{2}$.
\end{proof}

\begin{example}\label{exP3'}
Consider the arrangement $\Asr$ in $\Pbb^3$ defined by the equation $xyz(x+y)=0$. The characteristic polynomial $\chi_{\hat{\Asr}}$ of the corresponding affine central hyperplane arrangement in $\Abb^4$ is $x^4 - 4x^3 + 5x^2 -2x = 0$. The $m_S$ values in this example are $3,2,2,2$ respectively, and $\mu(1) = \chi_{\hat{\mathscr{A}}}(0) = 0$.
\begin{figure}[!h]
\begin{tikzpicture}[scale=1.5] 
   \draw (1,1) -- (1,2) (1,1) -- (2,2);
   \draw (2,1) -- (1,2) (2,1) -- (3,2);
   \draw (3,1) -- (1,2) (3,1) -- (4,2);
   \draw (4,1) -- (3,2) (4,1) -- (4,2) (4,1) -- (2,2);      
   \draw (2.5,0) node[above]{1};
   \draw \foreach \x in {1,...,4}{
       (2.5,0) node[below]{$\mathbb{A}^4$} -- (\x,1) node[below]{-1} 
       (\x,2) -- (2.5,3) node[above]{-2}       
       }; 
   \draw (1,2) node[above]{2}; 
   \draw \foreach \y in {2,3,4}{
       (\y,2) node[above]{1}
       };
\end{tikzpicture}
\caption{The intersection lattice of the affine arrangement $xyz(x+y)=0$ in $\mathbb{A}^4$ and the values of its M\"{o}bius function. The node marked with -2 is the only 1-dimensional element in $L(\mathscr{A})$, and there is no 0-dimensional element in $L(\mathscr{A})$.}
\end{figure}
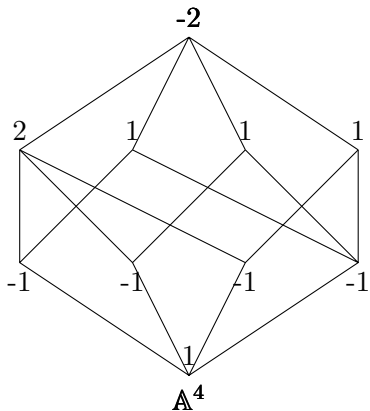

By theorem \ref{HMP3'}, we have
\begin{equation*}
(i_{\Asr,\Pbb^3})_*M_y(\Asr) = \Big((6y-1)h^2 + (-2y^2-21y+1)h^3\Big) \cap [\Pbb^3].
\end{equation*}

\end{example}

\bibliographystyle{alpha}
\bibliography{reference}

\end{document}